\documentclass{jgcc} 

\keywords{two-step nilpotent group, partially commutative group, graph group, exponential group, centralizer, centroid of group.}

\usepackage{hyperref}
\usepackage{orcidlink}

\usepackage{amssymb}
\theoremstyle{plain} 

\begin{document}
		
		\title[Centroids of Graph 2-Nilpotent Groups]{Centroids of Partially \\
			Commutative Two-Step Nilpotent Groups}

		\author[I.~M.~Buchinskiy]{Ivan M. Buchinskiy\, 
			\orcidlink{0000-0001-8637-9127} 
		}	
		\address{Omsk Department of Sobolev Institute of Mathematics, Siberian Branch of the Russian Academy of Sciences, Omsk, 644099, Russian Federation}	
		\email{buchvan@mail.ru}  
		\thanks{\textit{2020 Mathematics Subject Classification.} Primary 20F18; Secondary 20F36, 05C25.} 
		\thanks{\textit{UDC.} 512.544.33.} 

		%
		%
		%
		%

		
		
		
		
		\begin{abstract}
			\noindent In this paper, we describe the centroid (according to S. Lioutikov and A. Myasnikov) of an arbitrary finitely generated graph two-step nilpotent Hall exponential group over a binomial domain whose commutativity graph has no central vertices. Based on this description, we prove that if the non-commutativity graph of such a group has exactly one connected component and the group is non-Abelian, then the group is rigid.
		\end{abstract}


	\maketitle

	\section*{Introduction}\label{S:one}
	
	The concept of an \textit{exponential group} goes back to the idea of extending exponents in groups, which was originally developed by Ph. Hall, A.I. Mal'tsev, and R. Lyndon. Ph. Hall \cite{PhHall} began to consider exponents in a \textit{binomial domain} for nilpotent groups. A.I. Mal'tsev used the completion of the ring $\mathbb{Z}$ to $\mathbb{Q}$ to obtain the solvability of the equations $x^n = g$ for all $n \in \mathbb{Z}$. R. Lyndon \cite{Lyndon} defined the \textit{$R$-exponential groups} as groups satisfying axioms (1)-(3) from the \textit{Definition \ref{definition:ExpGroup}} given below. A.G. Myasnikov and V.N. Remeslennikov \cite{MyasRemesl} refined the concept of an exponential group by introducing an additional axiom (axiom (4) from the \textit{Definition \ref{definition:ExpGroup}}). Subsequently, such groups were called MR-groups. Throughout this paper, by the terms ``$R$-group'' and ``$R$-exponential group'' we will mean the MR-group over a ring $R$.
	
	In \cite{LioutMyas}, S. Lioutikov and A. Myasnikov introduced the concept of the \textit{centroid of a group $G$} as the set of all normal quasi-endomorphisms acting on $G$. In the same paper, it was shown that the centroid of $G$ is an associative ring with identity, and that it is a maximal ring of scalars whose action on $G$ satisfies all axioms of exponential groups \cite{MyasRemesl}. Note that K.N. Ponomarev in \cite{Ponomar} introduced his own definition of the centroid of a group, which is not equivalent in general to the definition in \cite{LioutMyas}. Both definitions are intended to generalize the concept of centroid from ring theory. As was shown in \cite{Ponomar}, these centroids coincide on some classes of groups (for example, on the class of torsion-free nilpotent groups). In the current work, we will use only the definition from \cite{LioutMyas}.
	
	In \cite{LioutMyas}, the centroids of finitely generated free nilpotent $R$-exponential groups and groups of unitriangular matrices $UT_n(R)$, where $R$ is a binomial domain, were described. Using the description obtained, S. Lioutikov and A. Myasnikov proved the rigidity of such groups and, thus, positively solved the problem of F. Grunewald and D. Segal from \cite{GrunSeg}. In addition, in \cite{LioutMyas}, the centroids of CSA-groups and free products of groups were described. Note that the centroid of an arbitrary Abelian group $A$ coincides with the endomorphism ring of $A$.
	
	The main goal of this paper is to study the rigidity of finitely generated partially commutative two-step nilpotent Hall $R$-exponential groups (meaning the Hall $R$-completion; see, for example, \cite{PhHall, LioutMyas}), where $R$ is a binomial domain, using the construction of the centroid of the group. The presence of central vertices in a commutativity graph complicates the description of the centroid structure of the corresponding partially commutative group. For this reason, in this article, we will restrict ourselves to the case where the graph has no central vertices.
	
	Partially commutative groups, or graph groups, or right-angled Artin groups, have many remarkable properties (convenient normal forms of elements, solvability of basic algorithmic problems, rich structure of subgroups, and much more; \cite{13, 15, MischTimosh, MischTreier, RemeslTreier}), the list of which continues to grow.
	
	The main result of the paper is \textit{Theorem \ref{theorem:rigidity}} from Section \ref{section:rigidity}, which shows the rigidity (see \textit{Definition \ref{def:rigidity}}) of a certain class of graph two-step nilpotent groups. This result was obtained using \textit{Theorem \ref{theorem:main}} from Section \ref{section:main}, which describes the centroid of an arbitrary finitely generated graph two-step nilpotent group $G_\Delta^R$, where $\Delta$ is a graph without central vertices, $R$ is a binomial domain. \textit{Theorem \ref{theorem:main}} establishes a dependence on the structure of the commutativity graph $\Delta$ and, in fact, generalizes the description of the centroid of a finitely generated free two-step nilpotent group from \cite[Theorem 5.2]{LioutMyas}. In addition, the structure of the centralizer of a block element in $G_\Delta^R$ was described in \textit{Lemma \ref{lemma:C_block_elem}} of Section \ref{section:C_block}. This result generalizes \cite[Lemma 2]{RemeslTreier}, previously proved for a field $R$ of characteristic zero.
	
	One of the founders of universal algebraic geometry \cite{DaniyarovaMyasnikovRemeslennikov_Monography}, V.N. Remeslennikov, suggested that the main way to study universal algebraic geometry over graph groups lies through the study of algebraic geometry over graphs. In this direction, it is proposed to describe the main algebraic-geometric properties of graph groups using the properties of the corresponding commutativity graphs (see, for example, \cite{Buch}). The results of this paper highlight some similar connections for the case of graph two-step nilpotent groups.
	
	We conjecture that many results of the present paper can be generalized to the case of an infinite number of generators: then in \textit{Theorem \ref{theorem:main}} the numbers $k$ and $l$ will be ordinals. However, this issue is beyond the scope of the current paper, and we plan to return to it in the future.
	
	\section{Preliminaries}
	
	Throughout what follows, let $[g, h] = g^{-1} h^{-1} g h$ denote the commutator of elements $g, h$ of the group $G$. 
	
	Let us fix an arbitrary associative ring with identity $A$ and a group $G$. Let $A$ have an action on $G$, that is, a mapping $G \times A \rightarrow G$. The result of the action of an element $a \in A$ on $g \in G$ will be denoted by $g^a$. Let $e$ denote the unit of the group $G$.
	
	\begin{defi}[\cite{MyasRemesl}]\label{definition:ExpGroup}
		A group $G$ is called an \textit{$A$-exponential group} (or an \textit{$A$-group}) if the action of $A$ on $G$ satisfies the following axioms:
		\begin{enumerate}[(1)]
			\item $\forall a \in A$ $\forall g \in G$ $g^{1} = g$, $g^{0} = e$, $e^a = e$;
			\item $\forall a,b \in A$ $\forall g \in G$ $g^{a+b} = g^a g^b$, $g^{ab} = \left(g^a\right)^b$;
			\item $\forall a \in A$ $\forall g,h \in G$ $\left(h^{-1} g h\right)^a = h^{-1} g^a h$;
			\item $\forall a \in A$ $\forall g,h \in G$ $[g, h] = e \rightarrow (gh)^a = g^a h^a$.
		\end{enumerate}
	\end{defi}
	\noindent
	For any $A$-group $G$, it is easy to see using axioms (1)-(2) that for any $k \in A \cap \mathbb{Z}$ and $g \in G$
	\[
	g^k = 
	\begin{cases}
		\underbrace{g \cdots g}_k, &\text{if } k > 0,\\
		e, &\text{if } k = 0, \\
		\underbrace{g^{-1} \cdots g^{-1}}_{-k} = (\underbrace{g \cdots g}_{-k})^{-1}, &\text{if } k < 0.
	\end{cases}
	\]
	
	On the way to defining the centroid of a group, let us recall some necessary concepts and results from \cite{LioutMyas}.
	
	Let $P(G)$ be the set of all mappings from $G$ to itself. We define on $P(G)$ the operations $+$ and $\cdot$ as follows:
	\begin{itemize}
		\item $\forall \varphi,\psi \in P(G)$ $\forall g \in G$ $(\varphi + \psi)(g) = \varphi(g) \cdot \psi(g)$;
		\item $\forall \varphi,\psi \in P(G)$ $\forall g \in G$ $(\varphi \cdot \psi)(g) = \psi(\varphi(g))$.
	\end{itemize}
	Among all mappings from $P(G)$, we select the identity $1(g) = g$ and the trivial $0(g) = e$ for all $g \in G$. Throughout the paper, we will denote $\varphi(g)$ by $g^\varphi$.
	
	A mapping $\varphi \in P(G)$ is called \textit{normal} if 
	\[\left(h^{-1} g h\right)^\varphi = h^{-1} g^\varphi h\]
	for all $g, h \in G$. Denote by $N(G)$ the set of all normal mappings of $G$. 
	
	\begin{lem}[\cite{LioutMyas}]\label{lemma:N_G_commut}
		Let $\varphi, \psi \in N(G)$, $x, y \in G$. Then $[x, y] = e$ implies $[x^\varphi, y^\psi] = e$.
	\end{lem}
	
	A mapping $\varphi \in P(G)$ is called a \textit{quasi-endomorphism} if 
	\[[g, h] = e \rightarrow (gh)^\varphi = g^\varphi h^\varphi\]
	for all $g, h \in G$. Denote by $Q(G)$ the set of all quasi-endomorphisms of $G$.
	
	\begin{thm}[\cite{LioutMyas}]\label{theorem:Centroid}
		The set of mappings $\Gamma(G) = N(G) \cap Q(G)$ with operations $+$, $\cdot$ and distinguished mappings $0$ and $1$ is an associative ring with identity.
	\end{thm}
	
	\begin{defi}
		The ring $\Gamma(G)$ from \textit{Theorem \ref{theorem:Centroid}} is called the \textit{centroid of the group $G$}.
	\end{defi}
	
	Throughout the remainder of our paper, we assume that the action of $A$ on $G$ satisfies axioms (1)-(4) from \textit{Definition \ref{definition:ExpGroup}} and is \textit{faithful} (effective), i.\,e., for every nonzero element $a \in A$ there exists $g \in G$ such that $g^a \ne e$. In this case, $A$ is called the \textit{ring of scalars of $G$}. 
	
	\begin{prop}[\cite{LioutMyas}]\label{propos:centroid_max}
		Let $G$ be a group. Then
		\begin{enumerate}[1)]
			\item the centroid $\Gamma(G)$ is a ring of scalars of $G$; in particular, $G$ is a $\Gamma(G)$-group;
			\item $\Gamma(G)$ is the largest ring of scalars of the group $G$; i.\,e., if $B$ is a ring of scalars of $G$, then there is a canonical embedding in $B \hookrightarrow \Gamma(G)$.
		\end{enumerate}
	\end{prop}
	
	\begin{rem}\label{remark:comm_with_integers}
		Let $A \supseteq \mathbb{Z}$, $G$ be an $A$-group, $g \in G$, $\varphi \in \Gamma(G)$, $k \in \mathbb{Z}$. Then, in particular, using axiom (4),
		\[\left(g^k\right)^{\varphi} = \left(g^\varphi\right)^k.\]
	\end{rem}
	
	The centralizer of $g \in G$ is defined as the set $C(g) = \{h \in G \ | \ [g, h] = e\}$. The centralizer of the set $X \subseteq G$ is defined naturally: $C(X) = \bigcap_{g \in X} C(g)$.
	
	\begin{lem}[\cite{LioutMyas}]\label{lemma:C_invariant}
		Let $G$ be an $A$-group. Then for any $X \subseteq G$ and $a \in A$, the following inclusion holds:
		\[C(X)^a \subseteq C(X).\]
	\end{lem}
	In particular, \textit{Lemma \ref{lemma:C_invariant}} is valid if we consider the centroid of $\Gamma(G)$ as the ring $A$. Moreover, it follows directly from \textit{Lemma \ref{lemma:C_invariant}} that $g^a \in C(g)$ for any $g \in G$ and $a \in A$.
	
	A group $G$ is called \textit{two-step nilpotent} if $[g,h,t] = [[g,h],t] = e$ for all $g,h,t \in G$.
	By $Z(G)$ we will denote the center of the group $G$, and by $G'$ the commutator subgroup of the group $G$. If $G$ is a two-step nilpotent group, then obviously $G' \subseteq Z(G)$.
	
	\begin{lem}[\cite{LioutMyas}]\label{lemma:comms}
		Let $g,h \in G$, $\varphi \in \Gamma(G)$. If $[g,h] \in Z(G)$, then 
		\[[g^\varphi, h] = [g, h^\varphi] = [g, h]^\varphi.\]
	\end{lem}
	
	\begin{rem}\label{remark:comms_cons}
		Let $G$ be a two-step nilpotent $A$-group. Then for any $z \in G'$, $\varphi, \psi \in \Gamma(G)$, the following equality holds: 
		\[\left(z^\varphi\right)^\psi = \left(z^\psi\right)^\varphi.\]
		Indeed, let $z = \prod [g_i, h_i] \in G'$ for some $g_i, h_i \in G$. Then, using axiom (4) from \textit{Definition \ref{definition:ExpGroup}}, $z^\varphi = \left(\prod [g_i, h_i]\right)^\varphi = \prod \left([g_i, h_i]\right)^\varphi$. Thus, it is enough for us to check the equality being verified only for $z = [g, h]$ for some $g, h \in G$. Let us recall that in a two-step nilpotent group the condition $G' \subseteq Z(G)$ is satisfied. Then, using \textit{Lemma \ref{lemma:comms}},
		$\left(z^\varphi\right)^\psi = \left([g, h]^\varphi\right)^\psi = \left[g^\varphi, h\right]^\psi = [g^\varphi, h^\psi] = \left[g, h^\psi\right]^\varphi = \left([g, h]^\psi\right)^\varphi = \left(z^\psi\right)^\varphi$.
		
		In particular, by \textit{Proposition \ref{propos:centroid_max}} and \textit{Lemma \ref{lemma:comms}}, for any $g, h \in G$ and any $a, b \in A$, the equality $\left[g^a, h^b\right] = [g, h]^{a b} = [g, h]^{b a}$ holds.
	\end{rem}
	
	\begin{defi}[\cite{PhHall}]
		A commutative integral domain $R$ is called a \textit{binomial domain} if $R$ contains $\mathbb{Z}$ as a subring and for every element $\lambda \in R$ and every positive integer $n$, the ring $R$ has the following binomial coefficient:
		\[C_\lambda^n = \frac{\lambda (\lambda - 1) \cdots (\lambda - n + 1)}{n!}.\]
	\end{defi}
	
	\begin{defi}[\cite{PhHall}]\label{def:Hall_exp}
		Let $G$ be a nilpotent group, and $R$ be a binomial domain. $G$ is called a Hall $R$-exponential group (or Hall $R$-completion) if the following axioms are satisfied:
		\begin{enumerate}[(1)]
			\item $\forall \lambda \in R$ $\forall g \in G$ $g^{1} = g$, $g^{0} = e$, $e^\lambda = e$;
			\item $\forall \lambda,\mu \in R$ $\forall g \in G$ $g^{\lambda+\mu} = g^\lambda g^\mu$, $g^{\lambda\mu} = \left(g^\lambda\right)^\mu$;
			\item $\forall \lambda \in R$ $\forall g,h \in G$ $\left(h^{-1} g h\right)^\lambda = h^{-1} g^\lambda h$;
			\item $\forall \lambda \in R$ $\forall g_1, \ldots, g_n \in G \ $ $g_1^\lambda g_2^\lambda \cdots g_n^\lambda = \tau_1^\lambda(\overline{g}) \tau_2^{C_\lambda^2} (\overline{g}) \cdots \tau_c^{C_\lambda^c} (\overline{g})$, where $\overline{g} = (g_1, \ldots, g_n)$, $\tau_i(\overline{g})$ are the Petresco words, and $c$ is the nilpotency class of $G$. Let us recall that the $i$-th Petresco word $\tau_i(\overline{g})$ is determined recurrently from the relation $g_1^i g_2^i \cdots g_n^i = \tau_1^i(\overline{g}) \tau_2^{C_i^2}(\overline{g}) \cdots \tau_{i-1}^{C_i^{i-1}}(\overline{g}) \tau_i(\overline{g})$. In particular, $\tau_1(\overline{g}) = g_1 g_2 \cdots g_n$, and $\tau_2(\overline{g}) = \tau_1^{-2}(\overline{g}) g_1^2 g_2^2 \cdots g_n^2 = (g_1 g_2 \cdots g_n)^{-2} g_1^2 g_2^2 \cdots g_n^2$.
		\end{enumerate}
	\end{defi}
	\noindent
	It is easy to see that if the group $G$ is a Hall $R$-exponential group (\textit{Definition \ref{def:Hall_exp}}), then it is also an $R$-exponential group (\textit{Definition \ref{definition:ExpGroup}}). Although this is written in \cite{MyasRemesl}, we will show it in more detail. Let $\lambda \in R$, $g, h \in G$ and $[g, h] = e$. Then, by axiom (4) from \textit{Definition \ref{def:Hall_exp}}, $g^\lambda h^\lambda = (gh)^\lambda \tau_2^{C_\lambda^2} (g, h) \cdots \tau_c^{C_\lambda^c} (g, h)$. Let us show that $\tau_i(g, h) = e$ for all $i > 1$. For $i = 2$ it is obvious: $\tau_2(g, h) = (g h)^{-2} g^2 h^2 = e$. Let it already be proven that $\tau_{j}(g,h) = e$ for all $j < i$. Then $\tau_i(g, h) = \tau_{i-1}^{-C_i^{i-1}}(g,h) \cdots \tau_2^{-C_i^2} (g, h) \tau_1^{-i}(g,h) g^i h^i = \tau_1^{-i}(g,h) g^i h^i = (gh)^{-i} g^i h^i = e$. Thus, $(gh)^\lambda = g^\lambda h^\lambda$, and therefore axiom (4) from \textit{Definition \ref{definition:ExpGroup}} is satisfied.
	
	Note that the set of two-step nilpotent Hall $R$-groups is a variety, since it is defined only by identities in the language $\mathcal{L}_{gr, R} = \{\cdot^{(2)}, ^{{-1}^{(1)}}, e\} \cup \{^{\alpha^{(1)}} \ | \ \alpha \in R\}$.
	By $G^R$ we denote the Hall $R$-completion of a group $G$ (see, for example, \cite{PhHall, LioutMyas}). If it is clear from the context that $G$ is the Hall $R$-exponential group, then we will simply write $G$.
	
	Following \cite{13}, we will provide a definition of the main object studied in this work.
	
	\begin{defi}[\cite{RemeslTreier}]
		Let $\Delta$ be an undirected simple graph with vertex set $X = V(\Delta)$, $F(X)$ be a free group with generator set $X$, and let 
		\[E = \{[x_i, x_j] \in F(X) \ | \ x_i, x_j \in X \text{ and $x_i$ and $x_j$ are adjacent in the graph $\Delta$}\}.\]
		Then the group $G_\Delta^R$ with set of generators $X$ and set of definability relations $E$, considered as a group in the variety of two-step nilpotent Hall $R$-groups, is called a two-step nilpotent \textit{partially commutative} (or \textit{graph}) Hall $R$-\textit{group}. And the graph $\Delta$ is called the \textit{commutativity graph} of the two-step nilpotent partially commutative Hall $R$-group $G_\Delta^R$.
	\end{defi}
	
	It can be seen that every two-step nilpotent graph Hall $R$-group $G_\Delta^R$ is in fact the Hall $R$-completion of a two-step nilpotent graph $\mathbb{Z}$-group $G_\Delta^{\mathbb{Z}}$ (the standard definition of a graph group that is not related to Hall $R$-completion can be found, for example, in \cite{13}). Moreover, since the set of two-step nilpotent Hall $R$-groups is a variety, every graph Hall $R$-group $G_\Delta^R$ with a set of definability relations $E$ is a free object. That is, there exists an $R$-homomorphism from $G_\Delta^R$ to any two-step nilpotent Hall $R$-group with definability relations $E$.
	
	\begin{defi}[\cite{MyasRemesl}]
		Let $G$ be an $A$-group. An element $g \in G$ is called \textit{periodic} if $g^\alpha = e$ for some $\alpha \in A \setminus \{0\}$. A group $G$ that does not contain non-identity periodic elements is called an \textit{$A$-torsion-free} group.
	\end{defi}
	
	It is known that $G_\Delta^{\mathbb{Z}}$ is $\mathbb{Z}$-torsion-free (see, for example, \cite{13}). In the next remark, we note that the same is true for any $G_\Delta^R$.
	
	\begin{rem}\label{remark:torsionFree}
		Every two-step nilpotent graph Hall $R$-group $G_\Delta^R$ is an $R$-torsion-free group. Indeed, let us assume that $g \in G_\Delta^R \setminus G_\Delta'$ and $\alpha \in R \setminus \{0\}$ such that $g^\alpha = e$. Let us consider the free Hall $R$-group $F_1^R$ generated by a single element $y$. Since $F_1^R$ is Abelian, it is two-step nilpotent and satisfies the definability relations of the group $G_\Delta^R$. Thus, there is an $R$-homomorphism $\varphi : G_\Delta^R \rightarrow F_1^R$ such that $\varphi(g) = y$. Then $\varphi(g^\alpha) = \varphi(g)^\alpha = y^\alpha \ne e$ and, consequently, $g^\alpha \ne e$. To understand why there is $R$-torsion-free on the elements of $\left(G_\Delta^R\right)' \setminus \{e\}$, it is sufficient to consider instead of $F_1$ a free two-step nilpotent Hall $R$-group $F_2$ (realized, for example, as $UT_3(R)$) and carry out similar reasoning for the case of $g \in \left(G_\Delta^R\right)' \setminus \{e\}$, i.\,e., for the case of $g = [f, h] \ne e$ for some $f, h \in G_\Delta^R \setminus \left(G_\Delta^R\right)'$. 
	\end{rem}
	
	Throughout what follows, by a partially commutative group $G_\Delta$ we will always mean, unless this causes confusion, a two-step partially commutative Hall $R$-group $G_\Delta^R$. In addition, the set of generators will be denoted by $X$.
	
	Let $\Delta$ be the commutativity graph of some graph group $G_\Delta$. Obviously, if $\Delta$ is not a clique, then for some generator $x_i$ in $G_\Delta$ there is a generator $x_j$ that does not commute with it. If $\Delta$ is a clique, then $G_\Delta$ is an Abelian $R$-group. Also note that if $\Delta$ has exactly one vertex, then $G_\Delta$ is also Abelian. If $\Delta$ is the empty graph (i.\,e., a graph without edges), then $G_\Delta$ is a free two-step nilpotent Hall $R$-group.
	
	\textit{A non-commutativity graph} of a graph group $G_\Delta$ is a graph $\Delta'$ with vertex set $X$ such that vertices $x_i$ and $x_j$ are adjacent in $\Delta'$ if and only if $x_i$ and $x_j$ are not adjacent in $\Delta$.
	
	Let $g \in G_\Delta$ be an element of the form $g = x_1^{\alpha_1} \cdots x_n^{\alpha_n} z$, where $x_i \in X$, $z \in G_\Delta'$, $\alpha_i \in R \setminus \{0\}$, and if $x_i$ occurs in the form $g$ several times with the exponents $\alpha_{i_1}, \ldots, \alpha_{i_k}$, then $\alpha_{i_1} + \ldots + \alpha_{i_k} \ne 0$. Then we denote $\overline{g} = x_1^{\alpha_1} \cdots x_n^{\alpha_n}$. Clearly, the $\overline{g}$ notation is ambiguous, since an element $g = x_1^{\alpha_1} \cdots x_n^{\alpha_n} z$ can be previously rewritten, for example, as $g = x_2^{\alpha_2} x_1^{\alpha_1} x_3^{\alpha_3} \cdots x_n^{\alpha_n} z [x_1^{\alpha_1}, x_2^{\alpha_2}]$ and then as $\overline{g} = x_2^{\alpha_2} x_1^{\alpha_1} x_3^{\alpha_3} \cdots x_n^{\alpha_n}$. We, however, will use this notation only when writing in the commutator or in the context of a commutation relation. For a partially commutative two-step nilpotent $R$-group, it doesn't matter which version of the $\overline{g}$ notation we use: $[g, h] = [\overline{g}, h] = [g, \overline{h}] = [\overline{g}, \overline{h}]$ for any $g, h \in G_\Delta$.
	
	For an arbitrary element $g \in G_\Delta$, we define the set $\alpha(g) \subseteq X$ as the set of all generators $x_i$ of $G_\Delta$ from the representation $\overline{g}$. As will be shown below in \textit{Proposition \ref{propos:baza}}, such a definition of $\alpha(g)$ does not depend on which version of $\overline{g}$ is being considered.
	
	For $x \in X$, we define
	\[x^\perp = \{y \in X \ | \ y \text{ is adjacent to } x \text{ in the graph } \Delta\} \cup \{x\}.\]
	For non-empty $Y \subseteq X$, we define $Y^\perp = \bigcap_{x_i \in Y} x_i^\perp$. Then, for example, $X^\perp = X \cap Z(G_\Delta)$ is the set of all central vertices of the graph $\Delta$. Also, by definition we put that $\emptyset^\perp = X$.
	
	Recall \cite{AmMyas} that a subgroup $H \leq G$ is called an \textit{$R$-subgroup} of $G$ if it is invariant under raising to the $R$-exponent, and a subset $X \subseteq G$ \textit{$R$-generates} an $R$-subgroup $\langle X\rangle_R$, which is the smallest $R$-subgroup of $G$ containing $X$. Everywhere below, by $\langle X\rangle$ we will mean an $R$-subgroup $R$-generated by $X$, and by subgroups we will mean $R$-subgroups.
	
	Let $Y \subseteq X$, $\Delta(Y)$ be the induced subgraph of $\Delta$ on the vertices of $Y$. Then by $\langle Y \rangle$ we denote the $R$-subgroup of $G_\Delta$ with the set of generators $Y$ and defining relations constructed from $\Delta(Y)$.
	
	\begin{rem}
		Let $\varphi, \psi \in \Gamma(G)$, $g, h \in G$. Then $[g, h] = e$ implies $[g^\varphi, h^\psi] = e$. Indeed, this follows from \textit{Lemma \ref{lemma:N_G_commut}}.
	\end{rem}
	
	\begin{defi}[\cite{13}]
		An element $g \in G_\Delta$ is called a \textit{block} if $\Delta'(\alpha(g))$ is a connected graph.
	\end{defi}
	\noindent
	Note that if $w$ is a block element and $|\alpha(w)| > 1$, then $\alpha(w) \cap \alpha(w)^\perp = \emptyset$.
	
	It is obvious that if $g = x_1^{\alpha_1} \cdots x_n^{\alpha_n} z$ is a block element, then it can be rewritten as a word $g = x_{i_1}^{\beta_{i_1}} \cdots x_{i_r}^{\beta_{i_r}} z'$, where $z' \in G_\Delta'$, and the generators $x_{i_1}, \ldots, x_{i_r} \in X$ are chosen in such an order that the path $P = x_{i_1} \ldots x_{i_r}$ in the graph $\Delta'(\alpha(g))$ passes through all vertices of $\Delta'(\alpha(g))$, and then as $\overline{g} = x_{i_1}^{\beta_{i_1}} \cdots x_{i_r}^{\beta_{i_r}}$. By default, we will assume that any block element is already represented this way.
	
	Let $g \in G_\Delta$. The element $\overline{g}$ can be written as $\overline{g} = q_1 \cdots q_n$, where each $q_i$ is a block element. We call $q_1 \cdots q_n$ the \textit{block decomposition of $g$}. Note that in such a decomposition $[q_i, q_j] = e$ for all $i,j = \overline{1, k}$. It is also easy to see that the block decomposition of an element may not be unique: it is enough to simply change the order of elements $q_i$. Let $\Delta'(\alpha(g)) = \Delta_1 \sqcup \ldots \sqcup \Delta_k$, where $\Delta_1, \ldots, \Delta_k$ are the connected components with more than one vertex in the non-commutativity graph $\Delta'$ of the graph $\Delta$. Note that if $x \in X \cap Z(G_\Delta)$, then $x \not\in \Delta_i$ for all $i = \overline{1, k}$. Then $\overline{g}$ can be written as $\overline{g} = w_1 \cdots w_k$, where $\Delta'(\alpha(w_i)) \subseteq \Delta_i$. At the same time, for some $w_i$ there may be $\overline{w}_i = e$. We call $w_1 \cdots w_k$ the \textit{decomposition of $g$ into block projections}, and $w_i$ the \textit{block projection of $g$ onto the component $\Delta_i$}. We also note that in such a decomposition $[w_i, w_j] = e$ for all $i,j = \overline{1, k}$. In general, a block projection may not be a block element. It is also easy to see that the block projection of an element may not be unique: in addition to the fact that it is possible to change the order of connected components, within the block decomposition of each $w_i$, it is possible to rearrange the elements $q_i$. 
	
	\section{Structure of the centralizer of the block element in $G_\Delta^R$}\label{section:C_block}
	
	From \cite[Lemma 1]{MischTreier} we know the general form of the centralizer of a block element of a graph two-step nilpotent $\mathbb{Q}$-group. In \cite[Lemma 2]{RemeslTreier} this result was generalized to the case of an arbitrary field of characteristic zero. Furthermore, in \cite[Lemma 1]{MischTimosh} this was generalized to the case of binomial Euclidean domains. On the other hand, for the case where the ring is an arbitrary binomial domain, similar results are not true: there are binomial domains (that are not Euclidean domains) such that when considering groups over which there exist two commuting block elements $w_1$ and $w_2$, for which $\alpha(w_1) = \alpha(w_2)$, and $w_2 \not\in \langle w_{1,0} \rangle$ where $w_{1,0}$ is some root element of $w_1$. Indeed, consider, for example, the $binom(\mathbb{Z}[\sqrt{-5}])$-group $G_\Delta$, where $binom(\mathbb{Z}[\sqrt{-5}])$ is the extension of $\mathbb{Z}[\sqrt{-5}]$ to a binomial domain by adding the missing formal binomial coefficients, and $\Delta$ is a graph without edges on two vertices $x_1$ and $x_2$. Then $\left[x_1^2 x_2^{1-\sqrt{-5}}, x_1^{1+\sqrt{-5}} x_2^3\right] = e$. Thus, there are binomial domains in which searching for the root element can lead to different results. In this section we generalize \cite[Lemma 1]{MischTimosh} to the case of Hall $R$-groups, where $R$ is an arbitrary binomial domain.
	
	Let $G_\Delta$ be the Hall $R$-exponential group, $R$ is an arbitrary binomial domain, $X = \{x_1, \ldots, x_n\}$ be the set of vertices in the graph $\Delta$.
	
	The following \textit{Proposition \ref{propos:baza}}, which is actually \cite[Proposition 1]{RemeslTreier}, we will present with a proof.
	
	\begin{prop}\label{propos:baza}
		The following statements are true:
		\begin{enumerate}
			\item the factor group $\overline{G}_\Delta = G_\Delta / G_\Delta'$ has the structure of a free module of rank $n$ over $R$ with a base of images of elements $x_1, \ldots, x_n$;
			\item the commutator subgroup $G_\Delta'$ of the group $G_\Delta$ has the structure of a free module over the ring $R$ with base $y_{i,j} = [x_i, x_j]$, where $y_{i,j} = [x_i, x_j] \ne e$ in the group $G_\Delta$ and $i < j$;
			\item any element in $G_\Delta$ is uniquely representable as
			$x_1^{\alpha_1} \cdots x_n^{\alpha_n} \prod y_{k,l}^{\beta_{k,l}},$
			where $x_i \in X$, $y_{k,l} = [x_k, x_l] \ne e$, $k < l$, and $\alpha_i, \beta_{i,j} \in R$.
		\end{enumerate}
	\end{prop}
	
	\begin{proof} \hfill
		\begin{enumerate}
			\item  This is obvious, since the set of definability relations of the $R$-group $G_\Delta$ belongs to the commutator subgroup of the corresponding free two-step nilpotent group.
		
			\item Assume that $\prod\limits_{i < j} y_{i,j}^{\alpha_{i,j}} = e$ and least one $\alpha_{k,l} \in R$ is nonzero. Let $F_2$ be a free two-generated two-step nilpotent Hall $R$-group with generators $a, b$. For definiteness, we can assume that $F_2$ is realized as $UT_3(R)$. In particular, the inequality $[a, b]^{\alpha_{k,l}} \ne e$ holds in $F_2$. There exists an $R$-homomorphism $\varphi : G_\Delta \rightarrow F_2$ such that $\varphi(x_k) = a$, $\varphi(x_l) = b$, and $\varphi(x_i) = e$ for all $i$ such that $l \ne i \ne k$. Then $\varphi(\prod\limits_{i < j} y_{i,j}^{\alpha_{i,j}}) = \varphi(y_{k,l}^{\alpha_{k,l}}) = \varphi(y_{k,l})^{\alpha_{k,l}} \ne e$. The second point follows from the resulting contradiction.
		
			\item Since $G_\Delta$ is two-step nilpotent, then $x_j^\beta x_i^\alpha = x_i^\alpha x_j^\beta [x_i, x_j]^{-\alpha \beta}$, where $i < j$. Consequently, any element of the group $G_\Delta$ can be written in the form described in the assertion of the third point. The uniqueness of the representation follows from the previous statements.
			\qedhere
		\end{enumerate}
	\end{proof}
	
	The following \textit{Proposition \ref{propos:det}} is a generalization of \cite[ Proposition 2]{MischTreier}.
	
	\begin{prop}\label{propos:det}
		Let $f = x_1^{\alpha_1} \cdots x_n^{\alpha_n} z_1$ and $g = x_1^{\beta_1} \cdots x_n^{\beta_n} z_2$, where $z_1, z_2 \in G_\Delta'$. Then 
		\[[f, g] = \prod y_{i,j}^{\delta_{i,j}},\] 
		where $y_{i,j} = [x_i, x_j] \ne e$, $\delta_{i,j} = \begin{vmatrix} 
			\alpha_i & \beta_i \\ 
			\alpha_j & \beta_j \notag 
		\end{vmatrix}$, $i < j$.
	\end{prop}
	
	\begin{proof}
		By \textit{Proposition \ref{propos:baza}}, $[f, g] = \prod y_{i,j}^{\gamma_{i,j}}$. Let us show this moment in more detail. 
		Since $G_\Delta$ is nilpotent of class two, the identity $[g^{\alpha} h^{\beta}, t^\gamma] = [g^{\alpha}, t^\gamma] [h^{\beta}, t^\gamma]$ holds for any $g, h, t \in G_\Delta$ and $\alpha, \beta, \gamma \in R$. Thus, 
		\begin{equation*}
			\begin{split}
				& [f, g] = [x_1^{\alpha_1} \cdots x_n^{\alpha_n}, x_1^{\beta_1} \cdots x_n^{\beta_n}] = \\
				& = [x_1^{\alpha_1}, x_1^{\beta_1} \cdots x_n^{\beta_n}] [x_2^{\alpha_2}, x_1^{\beta_1} \cdots x_n^{\beta_n}] \cdots [x_n^{\alpha_n}, x_1^{\beta_1} \cdots x_n^{\beta_n}] = \\
				& = [x_1^{\alpha_1}, x_1^{\beta_1}] [x_1^{\alpha_1}, x_2^{\beta_2}] \cdots [x_1^{\alpha_1}, x_n^{\beta_n}] [x_2^{\alpha_2}, x_1^{\beta_1}] \cdots [x_2^{\alpha_2}, x_n^{\beta_n}] \cdots [x_n^{\alpha_n}, x_1^{\beta_1}] \cdots [x_n^{\alpha_n}, x_n^{\beta_n}] = \\
				& = y_{1,2}^{\alpha_1 \beta_2} y_{1,3}^{\alpha_1 \beta_3} \cdots y_{1,n}^{\alpha_1 \beta_n} y_{1,2}^{-\alpha_2 \beta_1} y_{2,3}^{\alpha_2 \beta_3} \cdots y_{2,n}^{\alpha_2 \beta_n} \cdots y_{1,n}^{-\alpha_n \beta_1} \cdots y_{n-1, n}^{-\alpha_n \beta_{n-1}}.
			\end{split}
		\end{equation*}
		Thus, $\gamma_{i,j}$ from the equality $[f, g] = \prod y_{i,j}^{\gamma_{i,j}}$ is the same as $\alpha_i \beta_j - \alpha_j \beta_i = \delta_{i,j}$.
	\end{proof}

	\begin{defi}\label{def:W}
		For a block element $w \in G_\Delta$, having the representation $\overline{w} = x_{i_1}^{\beta_{i_1}} \cdots x_{i_r}^{\beta_{i_r}}$, where $\beta_{i_j} \ne 0$ for all $j = \overline{1, r}$, we define the set
		\begin{equation*}
			\begin{split}
				W(w) = &\{u \in G_\Delta \ | \ \overline{u} = x_{i_1}^{\gamma_{i_1}} \cdots x_{i_r}^{\gamma_{i_r}}, \text{ and } k_1 \cdot (\beta_{i_1}, \ldots, \beta_{i_r}) = k_2 \cdot (\gamma_{i_1}, \ldots, \gamma_{i_r})
				\\ 
				&\text{ for some } k_1, k_2 \in R 
				\text{ such that $k_1 \ne 0$ or $k_2 \ne 0$}\}.
			\end{split}
		\end{equation*}
		If $u \in W(w)$, then we will sometimes say that $u$ and $w$ are \textit{linearly dependent}. Note that $\overline{w} \in W(w)$ and $e = x_{i_1}^0 \cdots x_{i_r}^0 \in W(w)$. Also note that if $h \in W(w)$, then $\overline{h} = e$ (i.\,e., $h \in G_\Delta'$) or $\alpha(h) = \alpha(w)$.
	\end{defi}
	
	It is easy to see for $\overline{w}_1 \ne e \ne \overline{w}_2$ that if $w_1 \in W(w_2)$, then $w_2 \in W(w_1)$. It is also obvious that if $\overline{w}_1 \ne e$, $\overline{w}_2 \ne e$, $\overline{w}_3 \ne e$, $w_2 \in W(w_1)$, $w_3 \in W(w_2)$, then $w_3 \in W(w_1)$. Thus, the linear dependence of two block elements in $G_\Delta$ is an equivalence relation up to multiplication by elements from $G_\Delta'$. In particular, if $w_2 \in W(w_1)$, then $W(w_2) = W(w_1)$, and if $w_2 \not\in W(w_1)$, then $W(w_2) \cap W(w_1) = G_\Delta'$. 
	
	\begin{rem}\label{remark:C_gen}
		If $|\alpha(w)| = 1$ and $\alpha(w) = \{x_{i_s}\}$, then $C(w) = \langle x_{i_s}^\perp \rangle \cdot G_\Delta' \supseteq W(w) \cdot Z(G_\Delta)$. Indeed, if $\overline{w} = x_{i_s}^{\alpha}$, $w \not\in Z(G_\Delta)$, $[w, h] = e$, and $\overline{h} = x^{\beta_1}_{j_1} \cdots x^{\beta_m}_{j_m}$, then, using \textit{Remark \ref{remark:comms_cons}} and \textit{Proposition \ref{propos:baza}}, $[w, h] = [\overline{w}, \overline{h}] = [x^\alpha_{i_s}, x^{\beta_1}_{j_1} \cdots x^{\beta_m}_{j_m}] = [x^\alpha_{i_s}, x^{\beta_1}_{j_1}] 
		\cdots [x^\alpha_{i_s}, x^{\beta_m}_{j_m}] = [x_{i_s}, x_{j_1}]^{\alpha \beta_1} 
		\cdots [x_{i_s}, x_{j_m}]^{\alpha \beta_m}$. Thus, by \textit{Proposition \ref{propos:baza}}, for each $r$, $\beta_r = 0$ or $[x_{i_s}, x_{j_r}] = e$ (and, therefore, $x_{j_r} \in x_{i_s}^\perp$).
		
		From this, we have the general form of the centralizer for each generator $x_i \in X$ of the Hall $R$-group $G_\Delta$:
		\[C(x_i) = (\langle x_i \rangle \times \langle x_i^{\perp} \setminus \{x_i\} \rangle) \cdot G_\Delta' = \langle x_i^{\perp} \rangle \cdot G_\Delta'.\]
		In addition, we give as an example the centralizer of an arbitrary subset $A$ of generators:
		$C(A) = \langle A^{\perp} \rangle \cdot G_\Delta'.$
	\end{rem}
	
	Thus, the structure of centralizers of such elements differs from the case of $\mathbb{Q}$-groups only by replacing the ring with $R$. Further, we assume that $|\alpha(w)| > 1$.
	
	\begin{prop}\label{propos:W_commut}
		Let $w_1, w_2 \in G_\Delta$ be block elements, $|\alpha(w_1)| > 1$, and $|\alpha(w_2)| > 1$. Then, if $W(w_1) = W(w_2)$, then $[w_1, w_2] = e$.
	\end{prop}
	
	\begin{proof}
		Let $\overline{w}_1 = x_{i_1}^{\beta_{i_1}} \cdots x_{i_r}^{\beta_{i_r}}$, $r > 1$ and $\beta_{i_j} \ne 0$ for all $j = \overline{1, r}$, $\overline{w}_2 = x_{i_1}^{\gamma_{i_1}} \cdots x_{i_r}^{\gamma_{i_r}}$ and there exist $k_1, k_2 \in R$ such that $k_1 \beta_{i_j} = k_2 \gamma_{i_j}$. We show that $[w_1, w_2] = e$. By \textit{Proposition \ref{propos:det}}, $[w_1, w_2] = [\overline{w}_1, \overline{w}_2] = \prod y_{i_k,i_j}^{\delta_{i_k,i_j}}$, where $y_{i_k,i_j} = [x_{i_k}, x_{i_j}] \ne e$, $\delta_{i_k,i_j} = \begin{vmatrix}
			\beta_{i_k} & \beta_{i_j} \\
			\gamma_{i_k} & \gamma_{i_j} \notag
		\end{vmatrix}$, $k < j$. Note that $\delta_{i_k,i_j} = 0$ for all $i_k, i_j$. 
	\end{proof}
	
	Note that if we naturally extend the definition of a linear dependence relation (\textit{Definition \ref{def:W}}) to non-block elements and replace the condition \textquotedblleft $W(w_1) = W(w_2)$\textquotedblright \ with \textquotedblleft $w_1$ and $w_2$ are linearly dependent\textquotedblright \ in \textit{Proposition \ref{propos:W_commut}}, then the resulting proposition will be valid. This is unimportant for the main purpose of our article, so to simplify the technical details, we will work with \textit{Definition \ref{def:W}} as it is presented.
	
	The following proposition establishes a one-to-one correspondence between the set of linearly dependent elements with a block element $w$ and the set of all elements $h$ commuting with $w$ such that $\alpha(h) \subseteq \alpha(w)$.
	
	\begin{prop}\label{propos:block_w}
		Let $w \in G_\Delta$ be a block element and $|\alpha(w)| > 1$. If $h \in G_\Delta$ is such that $[w, h] = e$ and $\alpha(h) \subseteq \alpha(w)$ then $h \in W(w)$.
	\end{prop}
	
	\begin{proof}
		If $h \in G_\Delta'$ then the statement is obvious. Let $h \not\in G_\Delta'$, and $\overline{w} = x_{i_1}^{\beta_{i_1}} \cdots x_{i_r}^{\beta_{i_r}}$. Since $w$ is a block element, there exists a path $P = x_{i_1} \ldots x_{i_r}$ in the graph $\Delta'(\alpha(w))$ that passes through all vertices of $\Delta'(\alpha(w))$. Let $\overline{h} = x_{i_1}^{\gamma_{i_1}} \cdots x_{i_r}^{\gamma_{i_r}}$. Since $[w, h] = e$ and $[x_{i_1}, x_{i_2}] \ne e$, by \textit{Proposition \ref{propos:det}}, $\beta_{i_1} \gamma_{i_2} - \beta_{i_2} \gamma_{i_1} = 0$. Therefore, $\beta_{i_1} \gamma_{i_2} = \beta_{i_2} \gamma_{i_1}$. Take the edge $(x_{i_2}, x_{i_3})$ of the path $P$. Since $[x_{i_2}, x_{i_3}] \ne e$, by \textit{Proposition \ref{propos:det}}, $\beta_{i_2} \gamma_{i_3} - \beta_{i_3} \gamma_{i_2} = 0$, from which it follows that $\beta_{i_2} \gamma_{i_3} = \beta_{i_3} \gamma_{i_2}$. Reasoning similarly for the remaining edges of $P$, we obtain that $\beta_{i_j} \gamma_{i_{j+1}} = \beta_{i_{j+1}} \gamma_{i_j}$ for all $j = \overline{1, r-1}$. Let us show that there exist $k_1, k_2 \in R$ such that $k_1 \beta_{i_j} = k_2 \gamma_{i_j}$. Let $k_1 = \gamma_{i_1}$ and $k_2 = \beta_{i_1}$. Then it is obvious that $\gamma_{i_1} \beta_{i_1} = \beta_{i_1} \gamma_{i_1}$ and $\beta_{i_2} \gamma_{i_1} = \beta_{i_1} \gamma_{i_2}$. Let us show that $\beta_{i_3} \gamma_{i_1} = \gamma_{i_3} \beta_{i_1}$. Indeed, $\beta_{i_2} (\beta_{i_3} \gamma_{i_1} - \gamma_{i_3} \beta_{i_1}) = \beta_{i_2} \beta_{i_3} \gamma_{i_1} - \beta_{i_1} \beta_{i_3} \gamma_{i_2} = \beta_{i_3} (\beta_{i_2} \gamma_{i_1} - \beta_{i_1} \gamma_{i_2}) = 0$. Proceeding in a similar manner, we obtain that $\beta_{i_j} \gamma_{i_1} = \gamma_{i_j} \beta_{i_1}$ for all $j = \overline{1, r}$. Thus, $h \in W(w)$ and $\alpha(h) = \alpha(w)$.
	\end{proof}
	
	Let $w \in G_\Delta$. Denote by $A(w) = \langle \alpha(w)^\perp \rangle$. It is obvious that $[w, w_0] = e$ for any $w_0 \in A(w)$, and, moreover, if $w$ is a block element, $[v, w_0] = e$ for any $v \in W(w)$ and for any $w_0 \in A(w)$. We also note that if $w$ is a block element and $|\alpha(w)| > 1$, then $W(w) \cap A(w) \subseteq G_\Delta'$.
	
	\begin{lem}\label{lemma:C_block_elem}
		Let $w \in G_\Delta$ be a block element. Then the set $W(w)$ is an Abelian subgroup of $G_\Delta$, and, for $|\alpha(w)| > 1$,
		\[C(w) = W(w) \cdot A(w).\]
	\end{lem}
	
	\begin{proof}
		Let $\overline{w} = x_{i_1}^{\beta_{i_1}} \cdots x_{i_r}^{\beta_{i_r}}$. If $u, v \in W(w)$, then $[w, u v] = [\overline{w}, u v] = [\overline{w}, \overline{u}] [\overline{w}, \overline{v}] = e$. On the other hand, $\alpha(u v) \subseteq \alpha(w)$. Then, by \textit{Proposition \ref{propos:block_w}}, $u v \in W(w)$. As noted above, $e = x_{i_1}^0 \cdots x_{i_r}^0 \in W(w)$. Furthermore, if $\tilde{w} \in W(w)$, then $\tilde{w}^{-1} \in W(w)$. Thus, $W(w)$ is a subgroup of $G_\Delta$. The Abelian nature of $W(w)$ follows from \textit{Proposition \ref{propos:W_commut}}.
		
		Let $h \in G_\Delta$ and $[h, w] = e$. We show that $h \in W(w) \cdot A(w)$. Put $T = \alpha(w) \cup \alpha(w)^\perp$. Take $x_s \in \alpha(h)$ and let $x_s \not\in T$. Then there exists $x_{i_q} \in \alpha(w)$ such that $[x_{i_q}, x_s] \ne e$. Then, by \textit{Proposition \ref{propos:det}}, we have that
		\[[x_s, w] = [x_s, \overline{w}] = \left[x_s, x_{i_1}^{\beta_{i_1}} \cdots x_{i_r}^{\beta_{i_r}}\right] = [x_s, x_{i_q}]^{\beta_{i,q}} \prod_{j \in \{1, \ldots, r\} \setminus \{q\}} [x_s, x_{i_j}]^{\beta_{i_j}},\]
		where $\beta_{i,q} \ne 0$. Consequently, $[x_s, w] \ne e$, and, by \textit{Proposition \ref{propos:det}}, the coefficient of $[x_s, x_{iq}]$ in $[h, w]$ is $\gamma_s \beta_{i_q} \ne 0$. Thus, $\alpha(h) \subseteq T$.
		
		Since $w$ is a block element, $T = \alpha(w) \sqcup \alpha(w)^\perp$. We write the representative of $h$ as $\overline{h} = h_1 h_2$, where $\alpha(h_1) \subseteq \alpha(w)$ and $\alpha(h_2) \subseteq \alpha(w)^\perp$. Thus, $h_2 \in A(w)$. Note that $[w, h] = e$ if and only if $[w, h_1] = e$. Then, by \textit{Proposition \ref{propos:block_w}}, $h_1 \in W(w)$.
		
		Conversely, let $h \in W(w) \cdot A(w)$. Then $\overline{h} = \overline{v} \cdot \overline{w_0}$, where $v \in W(w)$, $w_0 \in A(w)$. By \textit{Proposition \ref{propos:W_commut}}, $[h, w] = [\overline{h}, \overline{w}] = [\overline{v} \cdot \overline{w_0}, \overline{w}] = e$.
	\end{proof}
	
	\begin{rem}
		As shown in \cite[Lemma 2]{RemeslTreier}, if $R$ is a field of characteristic zero, then
		\[C(w) = (\langle \overline{w} \rangle \times A(w)) \cdot G_\Delta'.\]
		Indeed, in this case, for any $k_1 \in R \setminus \{0\}$ and $k_2 \in R$, the equality $k_1 \beta_{i_j} = k_2 \gamma_{i_j}$ can be rewritten as $\beta_{i_j} = k_1^{-1} k_2 \gamma_{i_j}$, which actually means $W(w) = \langle \overline{w} \rangle \cdot G_\Delta'$. 
	\end{rem}
	
	Note that using \textit{Lemma \ref{lemma:C_block_elem}} we can describe the centralizer of an arbitrary element of the group $G_\Delta$, not just a block element. However, for the main result of the current work, given in the next section, we will only need \textit{Lemma \ref{lemma:C_block_elem}}.
	
	\section{Centroid of the group $G_\Delta^R$}\label{section:main}
	
	Let $R$ be a binomial domain, and $G_\Delta$ be a finitely generated partially commutative two-step nilpotent Hall $R$-group such that
	\begin{itemize}
		\item $X = \{x_1, \ldots, x_n\}$ -- the set of vertices of the graph $\Delta$,
		\item $\Delta_1, \ldots, \Delta_k$ -- all connected components with more than one vertex in the non-commutativity graph $\Delta'$ of the graph $\Delta$,
		\item $Y = \{y_1, \ldots, y_l\} = X \cap Z(G_\Delta)$ -- central vertices in the graph $\Delta$.
	\end{itemize}
	Note that $y_i \not\in \Delta_j$ for all $i$ and $j$. From \textit{Lemma \ref{propos:baza}} and the theory of block decompositions \cite{13}, it follows that any element $g \in G_\Delta$ has the form 
	\[g = w_1 \cdots w_k \cdot y_1^{\gamma_1} \cdots y_l^{\gamma_l} \cdot \prod\limits_{i = 1}^{k} \ \prod\limits_{\substack{x_s, x_r \in \Delta_i, \\ [x_s, x_r] \ne e, \\ s > r}} [x_s, x_r]^{\gamma_{s, r}},\] 
	where $w_1 \cdots w_k$ is the decomposition of $g$ into block projections and $\gamma_i, \gamma_{s,r} \in R$, and this representation is unique up modulo $G_\Delta'$ to a permutation of factors, including factors in block decompositions of block projections.
	
	In this section, we calculate the centroid of $G_\Delta$ for the case of $l = 0$. Before doing so, we will assume that $l$ is arbitrary until a certain point is explicitly noted.
	
	The following lemma gives a description of the images of generators of the group $G_\Delta$ with respect to arbitrary $\varphi \in \Gamma(G_\Delta)$.
	
	\begin{lem}\label{lemma:gen}
		For any $\varphi \in \Gamma(G_\Delta)$, there exist $\alpha_{\varphi, 1}, \ldots, \alpha_{\varphi, k} \in R$ such that for any $x_i \in X$ the following conditions hold:
		\begin{itemize}
			\item if $x_i \not\in Y$ and $x_i \in \Delta_j$, then $x_i^\varphi = x_i^{\alpha_{\varphi, j}} z_i$, where $z_i \in Z(G_\Delta)$;
			\item if $x_i \in Y$, then $x_i^\varphi \in Z(G_\Delta)$.
		\end{itemize}
	\end{lem}
	
	\begin{proof}
		Consider the action $\varphi \in \Gamma(G_\Delta)$ on an arbitrary generator $x_i$ of the group $G_\Delta$. By means of \textit{Lemma \ref{lemma:C_invariant}} and \textit{Remark \ref{remark:C_gen}}, the image of $x_i$ with respect to $\varphi$ in the general case has the form 
		\[x_i^\varphi = x_i^{\alpha_{\varphi,i}} t_{i,1}^{\beta_{\varphi,i,1}} \ldots t_{i,n_i}^{\beta_{\varphi,i,n_i}} \ \text{modulo $G_\Delta'$},\]
		where $\alpha_{\varphi,i}, \beta_{\varphi,i,q} \in R$, and $t_{i, q} \in x_i^\perp \setminus \{x_i\}$.
		
		Let $x_q \in X \setminus \{x_i\}$ be chosen arbitrarily. Since the group $G_\Delta$ is two-step nilpotent, we have $[x_i, x_q] \in Z(G_\Delta)$. Then, by \textit{Lemma \ref{lemma:comms}}, $\left[x_i^\varphi, x_q\right] = \left[x_i, x_q^\varphi\right]$. From this equality we write out the first and second commutators:
		\[\left[x_i^\varphi, x_q\right] = \left[x_i^{\alpha_{\varphi,i}} t_{i,1}^{\beta_{\varphi,i,1}} \cdots t_{i,n_i}^{\beta_{\varphi,i,n_i}}, x_q\right] = [x_i, x_q]^{\alpha_{\varphi,i}} \left[t_{i,1}^{\beta_{\varphi,i,1}} \cdots t_{i,n_i}^{\beta_{\varphi,i,n_i}}, x_q\right],\]
		\[\left[x_i, x_q^\varphi\right] = \left[x_i, x_q^{\alpha_{\varphi,q}} t_{q,1}^{\beta_{\varphi,q,1}} \cdots t_{q,n_q}^{\beta_{\varphi,q,n_q}} z_q\right] = [x_i, x_q]^{\alpha_{\varphi,q}} \left[x_i, t_{q,1}^{\beta_{\varphi,q,1}} \cdots y_{t,n_q}^{\beta_{\varphi,q,n_q}}\right].\]
		Thus,
		\[[x_i, x_q]^{\alpha_{\varphi,i}} \left[t_{i,1}^{\beta_{\varphi,i,1}} \cdots t_{i,n_i}^{\beta_{\varphi,i,n_i}}, x_q\right] = [x_i, x_q]^{\alpha_{\varphi,q}} \left[x_i, t_{q,1}^{\beta_{\varphi,q,1}} \cdots t_{q,n_q}^{\beta_{\varphi,q,n_q}}\right].\]
		Since the group $G_\Delta$ is partially commutative, by \textit{Proposition \ref{propos:baza}},
		\begin{equation}\label{eq:gen_1}
			[x_i, x_q]^{\alpha_{\varphi,i}} = [x_i, x_q]^{\alpha_{\varphi,q}},
		\end{equation}
		\begin{equation}\label{eq:gen_2}
			\left[t_{i,1}^{\beta_{\varphi,i,1}} \cdots t_{i,n_i}^{\beta_{\varphi,i,n_i}}, x_q\right] = \left[x_i, t_{q,1}^{\beta_{\varphi,q,1}} \cdots t_{q,n_q}^{\beta_{\varphi,q,n_q}}\right].
		\end{equation}
		If $[x_i, x_q] \ne e$, then, from equality (\ref{eq:gen_1}), we obtain that $\alpha_{\varphi,i} = \alpha_{\varphi,q}$. From this we note that if in the non-commutativity graph $\Delta'$ there is a path from vertex $x_i$ to vertex $x_r$, then $\alpha_{\varphi,i} = \alpha_{\varphi,r}$.
		Since in the representation $t_{i,1}^{\beta_{\varphi,i,1}} \cdots t_{i,n_i}^{\beta_{\varphi,i,n_i}}$ there is no generator $x_i$, and in $t_{q,1}^{\beta_{\varphi,q,1}} \cdots t_{q,n_q}^{\beta_{\varphi,q,n_q}}$ there is no generator $x_q$, then, from the equality (\ref{eq:gen_2}), we obtain
		$\left[t_{i,1}^{\beta_{\varphi,i,1}} \cdots t_{i,n_i}^{\beta_{\varphi,i,n_i}}, x_q\right] = \left[x_i, t_{q,1}^{\beta_{\varphi,q,1}} \cdots t_{q,n_q}^{\beta_{\varphi,q,n_q}}\right] = e.$
		Hence,
		$\beta_{\varphi,i,r} = 0$ for all $r$ such that $[t_{i, r}, x_q] \ne e$.
		Thus, in the representation for $x_i^\varphi$, only those $t_{i, r}$ can have nontrivial exponents for which the equality $[t_{i, r}, x_q] = e$ holds for all $x_q \in X \setminus \{x_i\}$, that is, $t_{i, r} \in Z(G_\Delta) \cap X$.
	\end{proof}
	
	Note that if $\Delta$ has no edges, then $x_i^\varphi = x_i^{\alpha_{\varphi,1}}$ modulo $G_\Delta'$ for all $i$. In particular, this follows from \cite[Lemma 5.2]{LioutMyas}.
	
	\begin{rem}\label{remark:gen_for_cyclic}
		Using \textit{Remark \ref{remark:C_gen}} and repeating the arguments of the proof of \textit{Lemma \ref{lemma:gen}}, one can easily find out the image of any element of the cyclic $R$-subgroup $\langle x_i \rangle$, where $x_i \in X$, with respect to each element of $\Gamma(G_\Delta)$. Thus, for any $x_i \in X$, $r \in R$, $\varphi \in \Gamma(G_\Delta)$,
		\[(x_i^r)^\varphi = (x_i^\varphi)^r \text{ modulo $Z(G_\Delta)$}.\]
		From this, if $x_i \not\in Z(G_\Delta)$ and $x_i \in \Delta_j$, then $(x_i^r)^\varphi = x_i^{r \cdot \alpha_{\varphi, j}}$ modulo $Z(G_\Delta)$. Note that in this case, the exponents at the central vertices in $(x_i^r)^\varphi$ may, until proven otherwise, not be related to the exponents in $x_i^\varphi$. 
	\end{rem}
	
	Let $\langle y_1 \rangle, \ldots, \langle y_l \rangle$ are cyclic $R$-subgroups generated by the central generators of $G_\Delta$ (the central vertices of the graph $\Delta$), $R^l$ is a set of all tuples of length $l$ with elements from $R$. Let's denote $Y_{cycl} = \langle y_1 \rangle \cup \ldots \cup \langle y_l \rangle$. Let us prove the following lemma, which describes the image of any element of the group $G_\Delta$ with respect to an arbitrary $\varphi \in \Gamma(G_\Delta)$.
	
	\begin{lem}\label{lemma:elems}
		For any $\varphi \in \Gamma(G_\Delta)$, there exist $\alpha_{\varphi,1}, \ldots, \alpha_{\varphi,k} \in R$ and a map $B_{\varphi, Z} : Y_{cycl} \rightarrow R^l$ such that
		\begin{enumerate}
			\item if $g \in G_\Delta$ and $g = w_{1} \cdots w_{k}$ is the decomposition of $g$ into block projections modulo $Z(G_\Delta)$, then $g^\varphi = w_{1}^{\alpha_{\varphi, 1}} \cdots w_{k}^{\alpha_{\varphi, k}}$ modulo $Z(G_\Delta)$, that is, the exponent of $w_i$ is determined by which of the connected components $\Delta_1, \ldots, \Delta_k$ the $\alpha(w_i)$ is included in;
			\item if $g \in Z(G_\Delta)$ and $g = y_1^{\gamma_{1}} \cdots y_l^{\gamma_{l}}$ modulo $G_\Delta'$, then $g^\varphi = y_1^{\beta_{\varphi,1}} \cdots y_l^{\beta_{\varphi,l}}$ modulo $G_\Delta'$, where the tuple $(\beta_{\varphi,1}, \ldots, \beta_{\varphi,l})$ is a sum of the tuples $B_{\varphi, Z}(y_1), \ldots, B_{\varphi, Z}(y_l)$, i.\,e., if $B_{\varphi, Z}(y_i^{\gamma_i}) = (\beta_{\varphi, y_i^{\gamma_i}, 1}, \ldots, \beta_{\varphi, y_i^{\gamma_i}, l})$, then $\beta_{\varphi,i} = \sum_{j = 1}^l \beta_{\varphi, y_j^{\gamma_j}, i}$ for all $i = \overline{1, l}$, that is, $\left(Z(G_\Delta)\right)^\varphi$ modulo $G_\Delta'$  is determined by $\left(Y_{cycl}\right)^\varphi$;
			\item if $g \in G_\Delta'$ and $g = u_1 \cdots u_k$, where $u_i = \prod_{x_s, x_r \in \Delta_i} [x_s, x_r]^{\gamma_{s,r}}$ and $\gamma_{s,r} \in R$, then $g^\varphi = u_1^{\alpha_{\varphi,1}} u_2^{\alpha_{\varphi,2}} \cdots u_k^{\alpha_{\varphi,k}}$, that is, the exponent of the basic commutator $[x_s, x_r]$ is determined by which of the connected components $\Delta_1, \ldots, \Delta_k$ contains $x_s$ and $x_r$.
		\end{enumerate}
	\end{lem}
	
	\begin{proof}
		\textbf{(1)} If $k = 0$ or $g \in Z(G_\Delta)$, then this statement is obviously true. We will further assume that $k > 0$ and $g \not\in Z(G_\Delta)$. Let $g = w_1 w_2 \cdots w_k$ modulo $Z(G_\Delta)$. Since $[w_i, w_j] = e$ for all $i,j = \overline{1, k}$, we have $g^\varphi = w_{1}^\varphi \cdots w_{k}^\varphi$ modulo $Z(G_\Delta)$. Next, we'll look at what each $w_i^\varphi$ looks like. Each block projection $w_i$ can be represented by a block decomposition $q_{i, 1} \cdots q_{i, n_i}$. Similarly, $w_i^\varphi = q_{i, 1}^\varphi \cdots q_{i, n_i}^\varphi$ modulo $Z(G_\Delta)$. To simplify the recording of further reasoning, without loss of generality we will assume that $w_i = w$, $w_j = e$ for all $j \ne i$, $n_i = 1$, and $q_{i, 1} = w$.
		If $|\alpha(w)| = 1$, then, using \textit{Remark \ref{remark:gen_for_cyclic}}, we get what $w^{\varphi}$ looks like modulo $Z(G_\Delta)$. Further we will assume that $|\alpha(w)| > 1$, $T = \alpha(w)^\perp \setminus Z(G_\Delta)$, and $\alpha(w) \subseteq \Delta_j$ for some $j$. Let us recall that $T \cap \alpha(w) = \emptyset$ and $T \supseteq \bigsqcup_{s \ne j} \Delta_s$. Note that two cases are possible: if $\alpha(w) = \Delta_j$, then $T = \bigsqcup_{s \ne j} \Delta_s$, and there may be $T \cap \Delta_j \ne \emptyset$ otherwise. We will consider both of these cases simultaneously. By \textit{Lemma \ref{lemma:C_invariant}} and \textit{Lemma \ref{lemma:C_block_elem}}, 
		\[w^\varphi = \tilde{w} \prod_{t_{s} \in T} t_{s}^{\gamma_{\varphi, s}} \ \ \text{modulo $Z(G_\Delta)$,}\]
		where $\tilde{w} \in W(w)$ and $\gamma_{\varphi, s} \in R$. In the case of a non-empty set $T$, we show that $\gamma_{\varphi, s} = 0$ for any $t_s \in T$. Let us recall that we assumed that $\alpha(w) \subseteq \Delta_j$. Since each $t_s \not\in Z(G_\Delta)$, there exists a generator $h_s \in X \setminus Y$ such that $[h_s, t_s] \ne e$. Let $h_s \in \Delta_r$ and possibly $r = j$.
		Then 
		\[[h_s, w^\varphi] = \left[h_s, \tilde{w} \prod_{t_{p} \in T \cap \Delta_r} t_{p}^{\gamma_{\varphi, p}}\right],\]
		and, on the other side, using \textit{Lemma \ref{lemma:comms}} and \textit{Lemma \ref{lemma:gen}}, $[h_s, w^\varphi] = [h_s^\varphi, w] = [h_s^{\alpha_{\varphi, j}}, w] = [h_s, w]^{\alpha_{\varphi, j}}$. Having described the commutators $[h_s, \tilde{w} \prod_{t_{p} \in T \cap \Delta_r} t_{p}^{\gamma_{\varphi, p}}]$ and $[h_s, w]^{\alpha_{\varphi, j}}$ through the basic commutators, we will come to the conclusion that $\gamma_{\varphi,p} = 0$ for all $p$ such that $t_{p} \in \Delta_r$ and $[h_s, t_p] \ne e$. Note that $t_s$, according to which $h_s$ was chosen, is also present among such $t_p$. Since $t_s \in T$ can be considered arbitrary, we come to the conclusion that
		$w^\varphi = \tilde{w}$ modulo $Z(G_\Delta)$. 
		
		Let us now show that $\tilde{w} = w^{\alpha_{\varphi, j}}$ modulo $Z(G_\Delta)$.
		Since $\tilde{w} \in W(w)$, then for some $\zeta, \xi \in R$ (simultaneously not equal to zero)
		\[\left(w^\varphi\right)^\zeta = w^\xi \ \ \text{modulo $Z(G_\Delta)$.}\]
		Let $x_s \in \alpha(w)$. By \textit{Lemma \ref{lemma:comms}}, $\left[w^\varphi, x_s^\zeta\right] = \left[w, \left(x_s^\zeta\right)^\varphi\right]$. We rewrite the second commutator using \textit{Lemma \ref{lemma:gen}}:
		\[\left[w, \left(x_s^\zeta\right)^\varphi\right] = \left[w^\varphi, x_s^\zeta\right] = \left[w^\varphi, x_s\right]^\zeta = \left[w, x_s^\varphi\right]^\zeta = \left[w, x_s^{\alpha_{\varphi,j}}\right]^{\zeta} = [w, x_s]^{\alpha_{\varphi,j} \zeta}.\]
		Note that if $\tilde{w} = e$ modulo $Z(G_\Delta)$, i.\,e., $\xi = 0$, then $\zeta \ne 0$ and, by \textit{Lemma \ref{lemma:comms}}, $e = [w^{\varphi}, x_s] = [w, x_s]^{\alpha_{\varphi,j} \zeta}$. Then it can only be that $\alpha_{\varphi, j} = 0$ and hence $w^\varphi = w^{\alpha_{\varphi, j}}$ modulo $Z(G_\Delta)$.
		Let $\tilde{w} \ne e$ modulo $Z(G_\Delta)$. Then
		\[[w, x_s]^{\alpha_{\varphi,j} \zeta} = \left[w^\varphi, x_s^\zeta\right] = \left[\left(w^\varphi\right)^\zeta, x_s\right] = \left[w^\xi, x_s\right] = [w, x_s]^\xi.\]
		Thus, $\alpha_{\varphi,j} \zeta = \xi$ and $\left(w^\varphi\right)^\zeta = w^\xi = w^{\alpha_{\varphi,j} \zeta}$ modulo $Z(G_\Delta)$. Since the $R$-roots are unique in $G_\Delta / Z(G_\Delta)$, we obtain $w^\varphi = w^{\alpha_{\varphi,j}}$ modulo $Z(G_\Delta)$.
		
		\textbf{(2)} Since $Z(G_\Delta)$ is a free Abelian group, $g \in Z(G_\Delta)$ can be written as $g = y_1^{\gamma_{1}} \cdots y_l^{\gamma_{l}}$ modulo $G_\Delta'$, where $\gamma_{1}, \ldots, \gamma_{l} \in R$. Then, using \textit{Lemma \ref{lemma:gen}} and \textit{Remark \ref{remark:gen_for_cyclic}}, modulo $G_\Delta'$ we have that
		\begin{equation*}
			\begin{split}
				& g^\varphi = \left(y_1^{\gamma_{1}} \cdots y_l^{\gamma_{l}}\right)^\varphi = \left(y_1^{\gamma_{1}}\right)^\varphi \cdots \left(y_l^{\gamma_{l}}\right)^\varphi = \\
				&= y_1^{\beta_{\varphi, y_1^{\gamma_1}, 1}} \cdots y_l^{\beta_{\varphi, y_1^{\gamma_1}, l}} \cdot y_1^{\beta_{\varphi, y_2^{\gamma_2}, 1}} \cdots y_l^{\beta_{\varphi, y_2^{\gamma_2}, l}} \cdots y_1^{\beta_{\varphi, y_l^{\gamma_l}, 1}} \cdots y_l^{\beta_{\varphi, y_l^{\gamma_l}, l}} = \\
				& = y_1^{\sum_{j = 1}^l \beta_{\varphi, y_j^{\gamma_j}, 1}} \cdots y_l^{\sum_{j = 1}^l \beta_{\varphi, y_j^{\gamma_j}, l}}.
			\end{split}
		\end{equation*}
		From here we set $B_{\varphi, Z}\left(y_i^{\gamma_i}\right) = \left(\beta_{\varphi, y_i^{\gamma_i}, 1}, \ldots, \beta_{\varphi, y_i^{\gamma_i}, l}\right)$ for all $i = \overline{1, l}$. If there is some $\gamma_i = 0$, then $y_i^{\gamma_i} = e$ and, consequently, we set $B_{\varphi, Z}(e) = (\underbrace{0, \ldots, 0}_l)$.
		
		\textbf{(3)} Let $g \in G_\Delta'$. To find $g^\varphi$, by \textit{Proposition \ref{propos:baza}}, it suffices to find the images of all basic commutators. Let $[x_s, x_r] \ne e$, where $x_s, x_r \in X$, and let $x_s, x_r \in \Delta_q$. From \textit{Lemma \ref{lemma:comms}} and \textit{Lemma \ref{lemma:gen}}, we obtain that
		$[x_s, x_r]^\varphi = \left[x_s, x_r^\varphi\right] = \left[x_s, x_r^{\alpha_{\varphi,q}}\right] = [x_s, x_r]^{\alpha_{\varphi,q}}$.
	\end{proof}
	
	The following lemma shows that the elements of $\Gamma(G_\Delta)$, in the general case, on $Z(G_\Delta)$ do not behave as $R$-homomorphisms modulo $G_\Delta'$, that is, the $\left(Z(G_\Delta)\right)^\varphi$ modulo $G_\Delta'$ is not determined only by $Y^\varphi$.
	
	\begin{lem}\label{lemma:not_homomorphism_on_center}
		Let $\gamma_1, \ldots, \gamma_l, \beta_1, \ldots, \beta_l \in (R \setminus \mathbb{Z}) \cup \{0\}$ such that for all $i = \overline{1, l}$ if $\gamma_i = 0$ then $\beta_i = 0$. Then there is $\varphi \in \Gamma(G_\Delta)$ such that $\left(y_i^{\gamma_i}\right)^\varphi = y_i^{\beta_i}$ for all $i = \overline{1, l}$ if and only if there are additive mappings $f_i : R \rightarrow R$ such that $f_i(\gamma_i) = \beta_i$ for all $i = \overline{1, l}$.
	\end{lem}
	
	\begin{proof}
		$\Leftarrow$. Following what has already been indicated, we recall that by $\langle Y \rangle$ we denote the $R$-subgroup of $G_\Delta$ with the set of generators $Y = \{y_1, \ldots, y_l\}$.
		Let us consider $\pi_Y : G_\Delta \rightarrow \langle Y \rangle$ be the homomorphism that kills every non-central vertex and fixes each $y_i$. It exists by the universal property, because $\langle Y \rangle$ is Abelian. Let $\psi(y_1^{a_1} \cdots y_l^{a_l}) = y_1^{f_1(a_1)} \cdots y_l^{f_l(a_l)}$ for $a_i \in R$. By first point of \textit{Proposition \ref{propos:baza}}, every element of $\langle Y \rangle$ has a unique form $y_1^{a_1} \cdots y_l^{a_l}$ for some $a_i \in R$, so $\psi$ is well defined. Set $\varphi = \psi \circ \pi_Y$. Then $\varphi$ is a group homomorphism, so it is a quasi-endomorphism. Its image is central, and $\varphi(h^{-1} g h) = \varphi(g) = h^{-1} \varphi(g) h$, so it is normal. Note that a mapping $\psi$ is only additive, not an $R$-homomorphism. This is fine because the definition of $\Gamma(G_\Delta)$ (normal quasi-endomorphisms) does not require $R$-linearity.
		
		$\Rightarrow$. For $t \in R$, let $f_i(t)$ be the exponent of $y_i$ in $\left(y_i^t\right)^\varphi$. The elements $y_i^s$ and $y_i^t$ commute, so $\left(y_i^{s+t}\right)^\varphi = \left(y_i^s\right)^\varphi \cdot \left(y_i^t\right)^\varphi$. Hence each $f_i$ is additive, and $f_i(\gamma_i) = \beta_i$.
	\end{proof}
	Additive mappings $f_i$ from the condition of \textit{Lemma \ref{lemma:not_homomorphism_on_center}} exist, for example, for the case when $R$ contains $\mathbb{Q}$ as a subring. Such additive mappings cannot be selected, for example, for the case when $p$ is a prime that is non-invertible in $R$, $\gamma = p \beta$, and $\beta \not\in p R$.
	
	Note that the first point of \textit{Lemma \ref{lemma:elems}} shows the image of a non-central element modulo $Z(G_\Delta)$. Therefore, if $w \in G_\Delta$ is a block element and $\alpha(w) \subseteq \Delta_i$ for some $i \in \{1, \ldots, k\}$, then
	\[w^\varphi = w^{\alpha_{\varphi, i}} \cdot y_1^{\beta_{\varphi, w, 1}} \cdots y_l^{\beta_{\varphi, w, l}} \cdot z,\]
	where the row $(\beta_{\varphi, w, 1}, \ldots \beta_{\varphi, w, l}) \in R^l$ corresponds to the element $w$, and $z \in G_\Delta'$. 
	Thus, let us further consider the mapping $B_\varphi : Q \cup Y_{cycl} \rightarrow R^l$, where $Q = \{v \in \langle q \rangle \ | \ $ $q \text{ is some block element such that if $|\alpha(q)| = 1$, then $\alpha(q) \not\subseteq Z(G_\Delta)$}\}$, such that $B_\varphi (y) = B_{\varphi, Z} (y)$ for any $y \in Y_{cycl}$, and $B_\varphi(q) = (\beta_{\varphi, q, 1}, \ldots, \beta_{\varphi, q, l})$ for any $q \in Q$. In addition, for the convenience of further recordings of block projection images, we define the mapping $B_\varphi$ onto the components $\Delta_i$ as follows: if $\Delta'(\alpha(w)) \subseteq \Delta_i$ and $w = q_1 \cdots q_n$ is a block decomposition then $B_\varphi(w) = \sum_{i = 1}^n B_\varphi(q_i)$.
	
	The following remark will not be used in this paper. It merely points out possible ways to generalize \textit{Lemma \ref{lemma:not_homomorphism_on_center}} and emphasizes the importance of the fact that the exponents of the central vertices in $(x_i^r)^{\varphi}$ may be unrelated to the exponents in $x_i^{\varphi}$ (see the last sentence in \textit{Remark \ref{remark:gen_for_cyclic}}).
	
	\begin{rem}\label{remark:not_homomorphism}
		Let $R$ be a field of characteristic zero, $w \in G_\Delta$ be a block element, and $\gamma_1, \ldots, \gamma_l, \beta_1, \ldots, \beta_l, \alpha, \beta'_1, \ldots, \beta'_l, \beta''_1, \ldots, \beta''_l \in (R \setminus \mathbb{Z}) \cup \{0\}$ such that for all $i = \overline{1, l}$ if $\gamma_i = 0$ then $\beta_i = 0$, and if $\alpha = 0$ then $\beta''_i = 0$. Let the following additive mappings also exist: 
		\begin{itemize}
			\item $f_i : R \rightarrow R$ such that $f_i(\gamma_i) = \beta_i$ for all $i = \overline{1, l}$, and
			\item $h_{i} : R \rightarrow R$ such that $h_{i}(1) = \beta_i'$ and $h_{i}(\alpha) = \beta_i''$ for all $i = \overline{1, l}$.
		\end{itemize}
		Then there is $\varphi \in \Gamma(G_\Delta)$ such that 
		\begin{itemize}
			\item $\left(y_i^{\gamma_i}\right)^\varphi = y_i^{\beta_i}$ for all $i = \overline{1, l}$, and
			\item $w^\varphi = w^{\alpha} \cdot y_1^{\beta'_1} \cdots y_l^{\beta'_l}$ and $\left(w^{\alpha}\right)^\varphi = w^{\alpha \cdot \alpha} \cdot y_1^{\beta''_1} \cdots y_l^{\beta''_l}$.
		\end{itemize}
		Indeed, this can be verified by conceptually following the proof of \textit{Lemma \ref{lemma:not_homomorphism_on_center}}.
	\end{rem}
	
	By writing $B_\varphi = 0$ and $B_{\varphi, Z} = 0$ we will mean that the mappings $B_\varphi$ and $B_{\varphi, Z}$ identically maps everything to the zero tuple.
	
	\begin{rem}\label{remark:B_tilde}
		The following statements are true:
		\begin{enumerate}
			\item if $\alpha_{\varphi, 1} = \ldots = \alpha_{\varphi, k} = 0$ and $B_{\varphi} = 0$, then $g^\varphi \in G_\Delta'$ for all $g \in G_\Delta$;
			\item if $B_{\varphi, Z} = 0$, then $g^\varphi \in G_\Delta'$ for all $g \in Z(G_\Delta) \setminus G_\Delta'$.
		\end{enumerate}
	\end{rem}
	\noindent
	To simplify further notation, for $B_\varphi(g) = (\beta_{\varphi, g, 1}, \ldots, \beta_{\varphi, g, l})$ we denote $B_{\varphi}(g)_i = \beta_{\varphi, g, i}$.
	
	\begin{rem}\label{remark:common_elems}
		As we noted at the beginning of this section, any element $g \in G_\Delta$ is written as $g = w_1 \cdots w_k y_1^{\gamma_1} \cdots y_l^{\gamma_l}z$, where $w_1 \cdots w_k$ is the decomposition of $g$  into block projections modulo $Z(G_\Delta)$, and $\gamma_1, \ldots, \gamma_l \in R$, $z \in G_\Delta'$. Then
		\[g^\varphi = w_1^{\alpha_{\varphi,1}} \cdots w_k^{\alpha_{\varphi,k}} y_1^{\beta_{\varphi,1}} \cdots y_l^{\beta_{\varphi,l}} \tilde{z},\] 
		where  
		$\beta_{\varphi,i} = \sum_{j = 1}^k B_\varphi(w_j)_i + \sum_{j = 1}^l B_\varphi\left(y_j^{\gamma_j}\right)_i$ for any $i = \overline{1, l}$, and $\tilde{z} \in G_\Delta'$. 
		Thus, the action of an arbitrary element of $\Gamma(G_\Delta)$ is completely determined by its action on the elements of $Q \cup Y_{cycl}$ and, by \textit{Lemma \ref{lemma:not_homomorphism_on_center}} and \textit{Remark \ref{remark:not_homomorphism}}, for some cases, cannot be determined by any proper subset in $Q \cup Y_{cycl}$.
		
		In addition, if $g \in Q \cup Y_{cycl}$, $\varphi \in \Gamma(G_\Delta)$, $s \in \mathbb{Z}$, then, using \textit{Remark \ref{remark:comm_with_integers}},
		\[B_{\varphi}(g^s) = s \cdot B_\varphi(g).\]
	\end{rem}
													
	It can be seen that the mappings $B_\varphi$ present for each $\varphi \in \Gamma(G_\Delta)$ (see \textit{Remark \ref{remark:not_homomorphism}} and \textit{Remark \ref{remark:common_elems}}) introduce sufficient technical complexity into the description of the structure of $\Gamma(G_\Delta)$. These mappings, in turn, arise from the central vertices $y_1, \ldots, y_l$ of the graph $\Delta$. Throughout the following, we will assume that $l = 0$, that is, there are no central vertices in the graph $\Delta$.
	
	The following definition is conceptually similar to \cite[Definition 5.1]{LioutMyas}.
	
	\begin{defi}\label{def:I}
		Let $I = \{\varphi \in \Gamma(G_\Delta) \ | \ \forall i = \overline{1, k} \ \ \alpha_{\varphi, i} = 0\}$.
	\end{defi}
													
	\begin{lem}\label{lemma:I_nilp}
		Let $\varphi \in \Gamma(G_\Delta)$. Then $\varphi$ is nilpotent if and only if $\varphi \in I$.
	\end{lem}
	
	\begin{proof}
		$\Leftarrow$. Directly follows from \textit{Definition \ref{def:I}} and the first and third points of \textit{Lemma \ref{lemma:elems}}. 
		
		$\Rightarrow$. Let $x \in X$ and $x \in \Delta_i$ for some $i$. Then, by \textit{Lemma \ref{lemma:gen}} and \textit{Remark \ref{remark:gen_for_cyclic}},
		\[x^{\varphi \cdot \varphi} = \left(x^\varphi\right)^\varphi = x^{\alpha_{\varphi, i} \cdot \alpha_{\varphi, i}} \text{ modulo $Z(G_\Delta)$}.\]
		Then, since $R$ is an integral domain, it turns out that $\alpha_{\varphi, i} = 0$ for any $i = \overline{1, k}$.
	\end{proof}
													
	\begin{lem}\label{lemma:I_images}
		Let $\varphi \in \Gamma(G_\Delta)$. Then $\varphi \in I$ if and only if $g^\varphi \in G_\Delta'$ for any $g \in G_\Delta$. In particular, if $\varphi \in I$, then  $z^\varphi = e$ for any $z \in G_\Delta'$.
	\end{lem}
	
	\begin{proof}
		$\Rightarrow$. This is obvious.
		
		$\Leftarrow$. It follows from the general form of the image of an element of the group $G_\Delta$ with respect to $\Gamma(G_\Delta)$ written in \textit{Lemma \ref{lemma:elems}} and \textit{Remark \ref{remark:common_elems}}.
	\end{proof}
													
	\begin{rem}\label{remark:nilp_two}
		For any $\varphi_1, \varphi_2 \in I$, the formula $\varphi_1 \cdot \varphi_2 = 0$ holds. Indeed, this follows from \textit{Lemma \ref{lemma:I_images}}.
	\end{rem}
	
	\begin{lem}\label{lemma:IJ}
		The set $I$ is a two-sided ideal in $\Gamma(G_\Delta)$.
	\end{lem}
													
	\begin{proof}
		Obviously, $0 \in I$ and if $\varphi \in I$, then $-\varphi \in I$.
		Let $\varphi, \psi \in I$, $g \in G_\Delta$. By \textit{Lemma \ref{lemma:I_images}}, we have $g^\varphi = z_1 \in G_\Delta'$ and $g^\psi = z_2 \in G_\Delta'$. Thus, $g^{\varphi + \psi} = g^\varphi g^\psi = z_1 z_2$ and $\varphi + \psi \in I$.
		
		Let $\varphi \in \Gamma(G_\Delta)$, $\psi \in I$, and $g \in G_\Delta$ be an arbitrary element. From \textit{Lemma \ref{lemma:I_images}}, it is obvious that $g^{\varphi \cdot \psi} = \left(g^\varphi\right)^\psi = z_1 \in G_\Delta'$, and therefore $\varphi \cdot \psi \in I$. On the other side, by, for example, the third point of \textit{Lemma \ref{lemma:elems}}, $g^{\psi \cdot \varphi} = \left(g^\psi\right)^\varphi = z_2^\varphi \in G_\Delta'$. Hence, $\psi \cdot \varphi \in I$.
	\end{proof}
													
	\begin{rem}
		Any non-trivial one-sided (left or right) ideal in $\Gamma(G_\Delta)$ intersects non-trivially with $I$. Indeed, let $J$ be a left (or right) sided ideal in $\Gamma(G_\Delta)$, $\varphi \in J \setminus I$ and $\alpha_{\varphi, i} \ne 0$ for some $i$. For $x \in X$ and $g \in G_\Delta$, let's define $\alpha_x(g)$ to be the exponent of $x$ in $g$. Let's consider the mapping $\psi(g) = [t_1, t_2]^{\alpha_{t_1}(g)}$ for $t_1, t_2 \in \Delta_i$ such that $[t_1, t_2] \ne e$. Then both $\psi \cdot \varphi$ and $\varphi \cdot \psi$ send $t_1$ to $[t_1, t_2]^{\alpha_{\varphi, i}} \ne e$, so they are non-zero, and they lie in $I$, which is a two-sided ideal in $\Gamma(G_\Delta)$. Whichever of them lies in $L$ is a non-zero element of $L \cap I$. 
	\end{rem}
													
	\begin{lem}\label{lemma:R_ideal}
		The following statements are true:
		\begin{enumerate}
			\item in $\Gamma(G_\Delta)$ there are $k$ subrings isomorphic to the binomial domain $R$, all pairwise intersections of which are equal to $\{0\}$;
			\item the sum $R^k = R_1 \oplus \cdots \oplus R_k$ (here $R_i = R$ for all $i = \overline{1, k}$) is a unital subring in $\Gamma(G_\Delta)$, and each $R_i$ is a two-sided ideal of $R^k$;
			\item the sets $R_i \oplus I = \{\varphi \in \Gamma(G_\Delta) \ | \ \alpha_{\varphi, j} = 0 \text{ for all } j \ne i\}$ are two-sided ideals in $\Gamma(G_\Delta)$.
		\end{enumerate}
	\end{lem}
													
	\begin{proof}
		(1) Let $\varphi \in \Gamma(G_\Delta)$ and $g \in G_\Delta$ be chosen arbitrarily. Let $g = w_1 \cdots w_k z$ (see \textit{Remark \ref{remark:common_elems}}), where $w_1 \cdots w_k$ is the decomposition of $g$ into block projections modulo $Z(G_\Delta)$, and $z \in G_\Delta'$. Then 
		\[g^\varphi = w_1^{\alpha_{\varphi,1}} \cdots w_k^{\alpha_{\varphi,k}} \tilde{z},\]
		where  
		$\tilde{z} \in G_\Delta'$. Let $\pi_i$, $i = \overline{1, k}$, the $R$-endomorphism of $G_\Delta$ that fixes the vertices of $\Delta_i$ and kills all other vertices. Let us denote by $\varphi_i$, $i = \overline{1, k}$, the mapping of the following form (recall that $g$ is chosen arbitrarily): $g \mapsto \pi_i(g)^{\alpha_{\varphi, i}}$. The mapping $\varphi_i$ is normal, since $(h^{-1} g h)^{\varphi_i} = (\pi_i(h)^{-1} \pi_i(g) \pi_i(h))^{\alpha_{\varphi, i}} = \pi_i(h)^{-1} \pi_i(g)^{\alpha_{\varphi, i}} \pi_i(h) = h^{-1} g^{\varphi_i} h$ for any $h \in G_\Delta$. Since $\alpha_{\varphi, i} \in R$ and $G_\Delta$ is a Hall $R$-exponential group (in particular, $R$-exponential group), for $h_1, h_2 \in G_\Delta$ such that $h_1 = w_1' \cdots w_k' z'$, $h_2 = w_1'' \cdots w_k'' z''$, $[h_1, h_2] = e$, and hence $[w_j', w_j''] = e$ for all $j = \overline{1, k}$,
		\begin{equation*}
			\begin{split}
				(h_1 h_2)^{\varphi_i} &= ((w_1' w_1'') \cdots (w_k' w_k'') (z' z''))^{\varphi_i} = (w_i' w_i'')^{\alpha_{\varphi, i}} \pi_i(z' z'')^{\alpha_{\varphi, i}} = \\
				&= (w_i')^{\alpha_{\varphi, i}} (w_i'')^{\alpha_{\varphi, i}} \pi_i(z')^{\alpha_{\varphi, i}} \pi_i(z'')^{\alpha_{\varphi, i}} = h_1^{\varphi_i} h_2^{\varphi_i}.
			\end{split}
		\end{equation*}
		Thus, $\varphi_i \in \Gamma(G_\Delta)$ for all $i = \overline{1, k}$. Besides this, it is easy to notice that $\{\varphi_i \ | \ \varphi \in \Gamma(G_\Delta)\} \cong R$ for all $i = \overline{1, k}$. Further, for convenience, we denote $R_i = \{\varphi_i \ | \ \varphi \in \Gamma(G_\Delta)\}$. It is not difficult to see that $R_i \cap R_j = \{0\}$ for all $i \ne j$. If $\varphi, \psi \in \Gamma(G_\Delta)$ then obviously $-\varphi_i, \varphi_i + \psi_i, \varphi_i \cdot \psi_i \in R_i$ for all $i = \overline{1, k}$. Note that
		the subrings $R_i$, unlike $I$, are neither left nor right sided ideals in $\Gamma(G_\Delta)$.
		
		(2) It is almost obvious from the first point. Directness follows by evaluating $\sum_{i = 1}^k \varphi_i = 0$ at a vertex of each $\Delta_j$. Identity in $R^k$ is $\pi_1 + \cdots + \pi_k$. In addition, $\pi_i \cdot \pi_j = 0$ for $i \ne j$, so $R_i \cdot R_j = 0$.
		
		(3) Follows directly from \textit{Lemma \ref{lemma:elems}}, \textit{Lemma \ref{lemma:IJ}}, the proof of the first point, and that $\varphi \mapsto (\alpha_{\varphi, 1}, \ldots, \alpha_{\varphi, k})$ is a ring homomorphism $\Gamma(G_\Delta) \rightarrow R^k$. So $R_i \oplus I$ is the preimage of an ideal $R_i$ in $R^k$.
	\end{proof}
													
	\begin{defi}
		Let's call the semidirect product $R^k \ltimes I$ the set of pairs $(s, i)$, where $s \in R^k, i \in I$, with 
		\[(s_1, i_1) + (s_2, i_2) = (s_1 + s_2, \ i_1 + i_2),\]
		\[(s_1, i_1) \cdot (s_2, i_2) = (s_1 \cdot s_2, \ s_1 \cdot i_2 + i_1 \cdot s_2 + i_1 \cdot i_2).\]
	\end{defi}
	The following facts about $R^k \ltimes I$ are true:
	\begin{itemize}
		\item $R^k \ltimes I$ is a ring with identity $(1, 0)$.
		\item The projection $p : (s, i) \mapsto s \in R^k$ is a surjective ring homomorphism. Its kernel $\{(0, i)\}$ is a copy of $I$, so it is a two-sided ideal in $R^k \ltimes I$, and $(R^k \ltimes I) / I \cong R^k$.
		\item The map $\sigma : R^k \ni s \mapsto (s, 0)$ is a ring homomorphism with $p \circ \sigma = id$. So the extension $0 \rightarrow I \rightarrow R^k \ltimes I \rightarrow R^k \rightarrow 0$ splits. Additively $R^k \ltimes I = R^k \oplus I$.
		\item Inside $R^k \ltimes I$, $(s, 0) \cdot (0, i) = (0, s \cdot i)$ and $(0, i) \cdot (s, 0) = (0, i \cdot s)$.
	\end{itemize}
													
	Summarizing all our reasoning, we formulate and prove the following theorem on the structure of the centroid for case $l = 0$.
													
	\begin{thm}\label{theorem:main}
		Let $R$ be a binomial domain, $G_\Delta$ be an arbitrary finitely generated graph two-step nilpotent Hall $R$-group with commutativity graph $\Delta$, and $\Delta$ is a graph without central vertices, i.\,e., $l = 0$. Then 
		\[\Gamma(G_\Delta) \cong R^k \ltimes I,\] 
		and $I^2 = 0$, where 
		\begin{itemize}
			\item $R^k$ is the direct sum of $k$ copies of $R$, and
			\item $k$ is the number of connected components 
			with more than one vertex
			in the non-commutativity graph $\Delta'$. 
		\end{itemize}
	\end{thm}
													
	\begin{proof}
		Nilpotency of $I$ follows from \textit{Lemma \ref{lemma:I_nilp}} and \textit{Remark \ref{remark:nilp_two}}.
		By \textit{Lemma \ref{lemma:IJ}}, $I$ is a two-sided ideal in $\Gamma(G_\Delta)$. In addition, $R^k \cap I = \{0\}$. 
		
		Let $\varphi \in \Gamma(G_\Delta)$ and $g \in G_\Delta$ be chosen arbitrarily, and $g = w_1 \cdots w_k z$. We retain the notation $\varphi_i$ ($i = \overline{1, k}$) for the mapping $w_1 \cdots w_k z \mapsto w_i^{\alpha_{\varphi, i}} \pi_i(z)^{\alpha_{\varphi, i}}$ from the proof of \textit{Lemma \ref{lemma:R_ideal}}.
		Let us define $\varphi' = \varphi - \sum_{i = 1}^k \varphi_i$. Since $\Gamma(G_\Delta)$ is a ring and each $\varphi_i \in \Gamma(G_\Delta)$, $\varphi' \in \Gamma(G_\Delta)$. Moreover, $\varphi' \in I$, since all $\alpha_{\varphi', i} = 0$. Then $\varphi = \sum_{i = 1}^k \varphi_i + \varphi'$, and this decomposition is unique because $R^k \cap I = \{0\}$.
		
		For $s \in R^k$ and $i \in I$, the map $s + i \mapsto (s, i)$ is multiplicative because $(s + i) \cdot (s' + i') = s s' + (si' + is' + ii')$ with $ss' \in R^k$ ($R^k$ is a subring by \textit{Lemma \ref{lemma:R_ideal}}) and the bracket in $I$ ($I$ is a two-sided ideal by \textit{Lemma \ref{lemma:IJ}}).
	\end{proof}
													
	The centroid $\Gamma(G_\Delta)$ is not the direct product $R^k \times I$ for two reasons:
	\begin{enumerate}
		\item The identity acts nontrivially on $I$. The identity $1$ of $\Gamma(G_\Delta)$ lies in $R^k$. In a direct product $R^k \times I$ we would have $R^k \cdot I = 0$, so every $\varphi \in I$ would satisfy $\varphi = 1 \cdot \varphi = 0$, forcing $I = 0$. But $I \ne 0$ for $k \ge 1$.
		\item For $k \ge 2$ the centroid is not even commutative. For $k = 1$ commutativity depends on $R$. It fails for example for $R = \mathbb{Q}(t)$ and the free group of rank $2$.
	\end{enumerate}							
													
	\begin{rem}\label{remark:ZplusI}
		$\Gamma(G_\Delta) \cong R \ltimes I$ if and only if the non-commutativity graph $\Delta'$ has exactly one connected component. In particular, if $\Delta$ is not connected, then $\Gamma(G_\Delta) \cong R \ltimes I$.
	\end{rem}
													
	The following examples demonstrate the application of \textit{Theorem \ref{theorem:main}} to specific groups.
	
	\begin{exa}\label{example:centroid_free}
		If $G_\Delta$ is a free two-step nilpotent Hall $R$-group, then $k = 1$, $l = 0$, and $I^2 = 0$. Then $\Gamma(G_\Delta) \cong R \ltimes I$ (see also \cite[Theorem 5.2]{LioutMyas}).
	\end{exa}
													
	\begin{exa}\label{example4}
		Let $G_\Delta = \langle x_1, x_2, x_3, x_4 \ | \ [x_1, x_3], [x_1, x_4], [x_2, x_3], [x_2, x_4] \rangle$ be a two-step nilpotent graph Hall $R$-group with commutativity graph $\Delta$ being a (2,2)-bipartite graph (the so-called $B_2$ graph \cite{Treier2020}). Then $k = 2$, $l = 0$ and hence $\Gamma(G_\Delta) \cong R^2 \ltimes I$.
	\end{exa}
	
	\begin{figure}[ht]
		\centering
		\includegraphics[width=45mm]{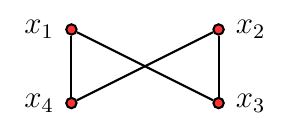}
		
		\footnotesize{Fig. 1. Commutativity graph $\Delta$ from \textit{Example \ref{example4}}.}
	\end{figure}
													
	Next, we show, for example, the structure of the centroid of partially commutative groups with commutativity graphs from a specific class. Following \cite{Treier2020}, let us recall the definition of a $B_q$-graph, where $q$ is a positive integer.
	Let $B_1$ be the graph consisting of two disjoint vertices. The series of graphs $B_1, B_2, \ldots, B_q, \ldots$ is defined by induction starting from $B_1$.
	Let $B_q$ already defined for some $q$. Then $V(B_{q+1}) = V(B_q) \cup \{u, v\}$ and both $u, v$ are connected with all vertices of $B_q$ and $u$ isn't connected with $v$. It is easy to see that in the $B_q$-graph the number $q$ is equal to the number of all connected components with more than one vertex in the non-commutativity graph (in the notation of this section, this is the number $k$). The following remark is aimed only at better understanding the structure of graph groups defined by $B_k$-graphs.
	
	\begin{rem}
		Let $R$ be a binomial domain. If $G_\Delta$ is a finitely generated graph two-step nilpotent Hall $R$-group with commutativity graph $\Delta$, where $\Delta$ is a $B_k$-graph, then $G_\Delta \cong F_2^k$, where $F_2$ is a free two-step nilpotent Hall $R$-group of rank $2$.
	\end{rem}
	
	Based on \textit{Example \ref{example4}}, we obtain 
													
	\begin{cor}
		Let $R$ be a binomial domain, $G_\Delta$ be an arbitrary finitely generated graph two-step nilpotent Hall $R$-group with commutativity graph $\Delta$, and $\Delta$ is a $B_k$-graph. Then $\Gamma(G_\Delta) \cong R^k \ltimes I$.
	\end{cor}
	
	\section{Rigidity of the group $G_\Delta$}\label{section:rigidity}
	
	Recall that by $G^R$ we denote the Hall $R$-completion of a group $G$ (see, for example, \cite{PhHall, LioutMyas}).
	
	\begin{defi}[\cite{GrunSeg}]\label{def:rigidity}
		A torsion-free nilpotent group $G$ is called \textit{rigid} if from $G^R \cong G^S$ it follows that $R \cong S$, where $R$ and $S$ are arbitrary binomial domains of characteristic $0$.
	\end{defi}
	
	The rigidity problem, formulated in \cite{GrunSeg}, is the following: is every finitely generated non-Abelian free nilpotent group rigid? In \cite[Theorem 6.1]{LioutMyas} it was answered in the affirmative. Rigidity results form an essential part of the F. Grunewald and D. Segal project \cite{GrunSeg} aimed at describing torsion-free nilpotent groups up to isomorphism. 
	
	In this section we show that there are infinitely many partially commutative two-step nilpotent rigid groups distinct from the non-Abelian free two-step nilpotent groups. We retain the notations from \textit{Theorem \ref{theorem:main}}.
	
	\begin{lem}\label{lemma:rigidity_t}
		Let $G_\Delta$ be a finitely generated partially commutative two-step nilpotent group, $\Delta$ be a commutativity graph without central vertices, and $R$ and $S$ be binomial domains. If $G_\Delta^R \cong G_\Delta^S$, then $R^k \cong S^k$.
	\end{lem}
	
	\begin{proof}
		The idea of the proof is based on the proof of \cite[Theorem 6.1]{LioutMyas}.
		Let $G_\Delta^R \cong G_\Delta^S$. Then $\Gamma(G_\Delta^R) \cong \Gamma(G_\Delta^S)$ and, by \textit{Theorem \ref{theorem:main}}, $R^k \ltimes I_1 \cong S^k \ltimes I_2$. From \textit{Lemma \ref{lemma:IJ}}, we know that $I_1$ and $I_2$ are ideals, and from \textit{Lemma \ref{lemma:I_nilp}}, they contain all nilpotent elements of $\Gamma(G_\Delta^R)$ and $\Gamma(G_\Delta^S)$, respectively. Let $f$ be an isomorphism between $R^k \ltimes I_1$ and $S^k \ltimes I_2$. Since the image of a nilpotent element under $f$ is nilpotent, $f(I_1) = I_2$. Thus, $f$ induces an isomorphism $\overline{f} : R^k \ltimes I_1 / I_1 \rightarrow S^k \ltimes I_2 / I_2$ and, therefore, $R^k \cong S^k$.
	\end{proof}
	
	\begin{thm}\label{theorem:rigidity}
		Every finitely generated non-Abelian graph two-step nilpotent group $G_\Delta$ whose non-commutativity graph $\Delta'$ consists of exactly one connected component is rigid. In particular, if $\Delta$ is a finite disconnected graph, then the graph two-step nilpotent group $G_\Delta$ is rigid.
	\end{thm}
	
	\begin{proof}
		This theorem follows directly from \textit{Lemma \ref{lemma:rigidity_t}}. On the other hand, note that by \textit{Theorem \ref{theorem:main}} and \textit{Remark \ref{remark:ZplusI}} (and also see \textit{Example \ref{example:centroid_free}}), $\Gamma(G_\Delta^R) \cong R \ltimes I \cong \Gamma(F_r^R)$, where $F_r^R$ is an arbitrary finitely generated free two-step nilpotent Hall $R$-group of rank $r > 1$. Thus, the proof of this theorem is identical to the proof of \cite[Theorem 6.1]{LioutMyas}.
	\end{proof}
	
	The next question naturally arises.
	
	\begin{qu}\label{problem1}
		Is every finitely generated non-Abelian partially commutative nilpotent group rigid?
	\end{qu}
	We also note a more general question, known to the author from A. Miasnikov.
	
	\begin{qu}
		Is every finitely generated non-Abelian nilpotent torsion-free group rigid?
	\end{qu}
	As noted in \cite{LioutMyas}, centroids and similar structures can be a useful tool to address this question.
	
	\section{Conclusion}\label{sec13}
	
	In this paper we related the structure of the centroid of an arbitrary finitely generated graph two-step nilpotent Hall $R$-exponential group $G_\Delta^R$, where $\Delta$ is a graph without central vertices and $R$ is a binomial domain, to the structure of its commutativity graph. From this we have succeeded in obtaining a class of graphs for which the groups are rigid.
	In addition, the formulated \textit{Remark \ref{remark:common_elems}} create some scope for the analysis of the centroid for an arbitrary case (when the commutativity graph can have central vertices). 
	
	In the future, we plan to focus on solving \textit{Question \ref{problem1}}.
	
	\section*{Acknowledgment}
	\noindent The research was funded in accordance with the state task of the IM SB RAS, project FWNF-2026-0033.	
	
	The author expresses gratitude to his supervisor, Alexander V. Treier, for valuable discussions on this work. 
	The author also sincerely thanks the anonymous reviewer for carefully reading the article and providing helpful suggestions for its improvement. The author also thanks Harmonic for the excellent Aristotle tool, which allows for automatic formal verification of results.
	

\end{document}